\documentclass{article}

\usepackage{amsfonts}
\usepackage{amsthm}
\usepackage{amsmath}
\usepackage{amssymb}
\usepackage{mathtools} 
\usepackage{authblk} 
\usepackage{xcolor}
\usepackage{soul} 
\usepackage{mathrsfs}
\usepackage{graphicx}
\usepackage{caption}
\usepackage{subcaption}
\usepackage{float}
\usepackage[top=1in, bottom=1in, left=1in, right=1in]{geometry}

\usepackage{amsfonts}
\usepackage{amsthm}
\usepackage{amsmath}
\usepackage{amssymb}
\usepackage{mathtools} 
\usepackage{authblk} 
\usepackage{xcolor}
\usepackage{caption}
\usepackage{subcaption}
\usepackage{graphicx}
\usepackage{float}
\usepackage{soul} 
\usepackage{mathrsfs}

\usepackage[top=1in, bottom=1in, left=1in, right=1in]{geometry}

\newtheorem{thm}{Theorem}[section]
\newtheorem{lem}{Lemma}[section]
\newtheorem{coroll}{Corollary}[section]

\theoremstyle{definition}
\newtheorem{defn}{Definition}[section]

\theoremstyle{remark}
\newtheorem{remark}{Remark}[section]

\definecolor{cucol}{rgb}{0,0,0.8}
\definecolor{afcol}{rgb}{1,0,0}

\numberwithin{equation}{section}

\begin{document}
	
		
		\title{A natural extension of Mittag-Leffler function associated with a triple infinite series}
		
		\date{}
		

		\author[]{Ismail T. Huseynov\thanks{ Email:\texttt{ismail.huseynov@emu.edu.tr}}}
		\author[]{Arzu Ahmadova\thanks{Email: \texttt{arzu.ahmadova@emu.edu.tr}} }
		\author[]{ Gbenga O. Ojo \thanks{Corresponding author.  Email: \texttt{ojo.gbenga@emu.edu.tr}}}
		\author[]{Nazim I. Mahmudov \thanks{Email: \texttt{nazim.mahmudov@emu.edu.tr}}}

		\affil[] {Department of Mathematics, Faculty of Arts and Sciences, Eastern Mediterranean University, Mersin 10, Gazimagusa, TRNC, Turkey}

		
		\maketitle
		

\begin{abstract}
\noindent We establish a new natural extension of Mittag-Leffler function with three variables which is so called "trivariate Mittag-Leffler function". The trivariate Mittag-Leffler function can be expressed via complex integral representation by putting to use of the eminent Hankel's integral. We also investigate Laplace integral relation and convolution result for a univariate version of this function. Moreover, we present fractional derivative of trivariate Mittag-Leffler function in Caputo type and we also discuss Riemann--Liouville type fractional integral and derivative of this function. The link of trivariate Mittag-Leffler function with fractional differential equation systems involving different fractional orders is necessary on certain applications in physics. Thus, we provide an exact analytic solutions of homogeneous and inhomogeneous multi-term fractional differential equations by means of a newly defined trivariate Mittag-Leffler functions.

\end{abstract}  
		

\section{Introduction and preliminaries}\label{Sec:intro}
\textbf{Special functions} are one of the powerful implements in presenting and describing some physical complex phenomena in fractional calculus \cite{Kiryakova-1}-\cite{Kalla et al.}.
Decades ago, the special function entitled
\textbf{Mittag-Leffler function (M-L)} has drawn an arising attention by many researchers due to its importance in solving differential and integral equations with fractional-order in science and engineering \cite{viscoelasticity}-\cite{cell}.

The classical M--L function which is a natural generalization of the exponential function was proposed by Mittag-Leffler in 1903 as a one-parameter function of one variable by using a single series,
\begin{equation}\label{ml-1}
E_{\alpha}(s)= \sum_{l=0}^{\infty}\frac{s^{l}}{\Gamma(l \alpha +1)},\quad \alpha \in \mathbb{C},  \Re(\alpha) >0, s\in\mathbb{C},
\end{equation}
and investigated its properties in \cite{ML-1}-\cite{ML-4}.
\begin{remark}[\cite{Diethelm}]
The M--L functions are often used in a form where the variable inside the brackets is not $s$ but a fractional power $r^{\alpha}$, or even a constant multiple $\lambda r^{\alpha}$, as follows:
\begin{equation*}
E_{\alpha}(\lambda r^{\alpha})= \sum_{l=0}^{\infty}\frac{\lambda^l r^{l\alpha}}{\Gamma(l \alpha +1)}, \quad \alpha \in \mathbb{C}, \Re(\alpha) >0, \lambda, r\in\mathbb{C}.
\end{equation*}
\end{remark}

A generalization of \eqref{ml-1}, particularly the two-parameter M--L function was presented by Wiman in 1905 and he determined this function \cite{Wiman-1,Wiman-2} as 

\begin{equation}\label{ml-2}
E_{\alpha,\beta}(s)= \sum_{l=0}^{\infty}\frac{s^{l}}{\Gamma(l \alpha +\beta)},\quad \alpha, \beta\in\mathbb{C}, \Re(\alpha) >0, s\in\mathbb{C},
\end{equation}
which has deeply studied in \cite{Humbert}-\cite{Humbert-Agarwal}.

A natural extension of \eqref{ml-2}, which is called the M--L with three-parameter was proposed by Prabhakar \cite{Prabhakar} in 1971  as

\begin{equation} \label{Prabf}
E_{\alpha,\beta}^{\delta}(s)= \sum_{l=0}^{\infty}\frac{(\delta)_l}{\Gamma(l \alpha +\beta)}\frac{s^{l}}{l!}, \quad \alpha,\beta,\delta \in\mathbb{C},\Re(\alpha)> 0,  s\in\mathbb{C},
\end{equation}
where $(\delta)_l$ is the Pochhammer symbol \cite{Rainville} denoting

\begin{equation*}
(\delta)_l=\frac{\Gamma(\delta + l)}{\Gamma(\delta)}= \begin{cases*}
1, \quad l=0,\delta \neq 0, \\
\delta (\delta + 1)\cdots(\delta + l - 1), \quad l \in \mathbb{N}.
\end{cases*}
\end{equation*} 
 This series \eqref{Prabf} is widely used for different applied problems like heat conduction equations with memory \cite{Garra-Garrapa}, electrical circuits \cite{Soubhia et.al.}, Langevin equations \cite{Camargo et.al.},  anomalous relaxation in dielectrics \cite{Oliveira et.al.} and fractional order time-delay systems \cite{Huseynov-Mahmudov,mhmdv,mahmudov-areen}.
 Note that
 \begin{equation*}
 E_{\alpha,\beta}^1(s)=E_{\alpha,\beta}(s),\quad E_{\alpha,1}(s)=E_{\alpha}(s),\quad E_1(s)=\exp(s).
 \end{equation*}
 It is interesting to note that extensions to two or three parameters are well known and thoroughly studied in textbooks \cite{Gorenflo,Paneva}, but these still involve single power series in one variable. Recently decades, a various type of extensions of M--L functions have been defined: namely "bivariate" and "multivariate" M--L functions. 

 A multi-variable analogue of generalized Mittag-Leffler type function is proposed by Saxena et al. \cite{saxena-kalla-saxena} in the form 
\begin{equation}\label{multivariate}
\mathcal{E}^{(\delta_1,...,\delta_m)}_{(\rho_1,...,\rho_m),\lambda}(s_1,...,s_m)=\sum\limits_{l_1,...,l_m=0}^{\infty}\frac{(\delta_1)_{l_{1}}...(\delta_m)_{l_{m}}}{\Gamma(\lambda + \sum\limits_{j=1}^{m}\rho_j l_j)}\frac{s_1^{l_1}... s_m^{l_m}}{l_1!...l_m!},
\end{equation}
 where $\lambda,\rho_j,\delta_j,s_j \in \mathbb{C},\Re(\rho_j) >0, j=1,...,m.$
 
 Another multivariable analogue of M--L function $E_{\alpha_1,...,\alpha_n}(s_1,...,s_n)$ of $n$ variables $s_1,...,s_n \in \mathbb{C}$ and $n$ parameters $\alpha_1,...,\alpha_n,\beta \in \mathbb{C}$ with $\Re(\alpha_j),\Re(\beta)>0,j=1,...,n$ is defined by
 
 \begin{equation}\label{multidimensional}
 E_{\alpha_1,...,\alpha_n}(s_1,...,s_n)=\sum\limits_{k=0}^{\infty}\sum^{l_1+...l_n=k}_{l_1\geq 0,...,l_n \geq 0}(k;l_1,...,l_n)\frac{\prod\limits_{j=1}^{n} s_j^{l_j}}{\Gamma(\beta+\sum\limits_{j=1}^{n}\alpha_j l_j)},
 \end{equation}
 where
 \begin{equation*}
 (k;l_1,...,l_n)\coloneqq \frac{k!}{l_1!\times...\times l_n!} \quad \text{with} \quad k=\sum_{j=1}^{n}l_j.
 \end{equation*}
 This function \eqref{multidimensional} is studied by Luchko et al. \cite{Luchko-Gorenflo} as an analytical solution of the Caputo type fractional differential equations (FDEs) with multi-orders.
 
 It is necessary to point out that various bivariate functions are rising as an extension of the M--L function: one of them is proposed by Özarslan et al. \cite{Ozarslan-Kurt} as below:
 \begin{equation}
 E^{(\rho)}_{\alpha,\beta,q}(u,v)=\sum\limits_{l=0}^{\infty}\sum_{p=0}^{\infty}\frac{(\rho)_{l+p}}{\Gamma(\beta+qp)\Gamma(\alpha +l)}\frac{u^l}{l!}\frac{v^{qp}}{p!},
 \end{equation} 
 under the conditions $\alpha,\beta,q,\rho \in \mathbb{C} \quad \text{with} \quad  \Re(\alpha),\Re(\beta)$, and $\Re(q)>0$.
 
 Another a new analogue of classical M--L function which is applied two variables are proposed by Fernandez et al. \cite{Fernandez-Kurt-Ozarslan} as follows:
\begin{equation}\label{bML4}
E_{\alpha,\beta,\gamma}^{\delta}(u,v)=\sum_{l=0}^{\infty}\sum_{p=0}^{\infty}\frac{(\delta)_{l+p}}{\Gamma(l\alpha+p\beta+\gamma)}\frac{u^lv^p}{l!p!},\quad \alpha, \beta, \gamma, \delta \in\mathbb{C},\quad \Re(\alpha), \Re(\beta)>0, u,v\in\mathbb{C}.
\end{equation}
To form a univariate version of \eqref{bML4}, we write $u=\lambda_{1}r^{\alpha}$ and $v=\lambda_{2}r^{\beta}$ and multiply by a power function:
\begin{equation}\label{bivtype}
r^{\gamma-1}	E_{\alpha,\beta,\gamma}^{\delta}(\lambda_{1}r^{\alpha},\lambda_{2}r^{\beta})=\sum_{l=0}^{\infty}\sum_{p=0}^{\infty}\frac{(\delta)_{l+p}}{\Gamma(l\alpha+p\beta+\gamma)}\frac{\lambda_{1}^l\lambda_{2}^p}{l!p!}r^{l\alpha+p\beta+\gamma-1}.
\end{equation}
It is worth noting that the univariate analogue of \eqref{bML4} is an exact analytical solution of FDEs system with  multi-orders which has been discussed in \cite{ismail-arzu}.

Now, we consider another special function which will be introduced later in Section \ref{sec:multi}. Let $\lambda_{i},\mu_{j} \in \mathbb{C}$, $a_{i},b_{j} \in \mathbb{R}$, for $i=1,2,\ldots,m$, and $j=1,2,\ldots,n$. Generalized Wright function or more appropriately Fox-Wright function $\prescript{}{m}{\Psi_{n}}(\cdot):\mathbb{C}\to \mathbb{C}$ is defined by
\begin{equation}\label{fox}
\prescript{}{m}{\Psi_{n}}(s)=\prescript{}{m}{\Psi_{n}}\left[\begin{array}{ccc}
(\lambda_{i},a_i)_{1,m} \\
(\mu_j,b_j)_{1,n} 
\end{array}\Big|s\right]=\sum_{l=0}^{\infty}\frac{\prod\limits_{i=1}^{m}\Gamma(\lambda_{i}+a_{i}l)}{\prod\limits_{j=1}^{n}\Gamma(\mu_{j}+b_{j}l)}\frac{s^{l}}{l!}.
\end{equation}
This Fox-Wright function was established by Fox \cite{fox} and Wright \cite{wright}. If the following condition is satisfied 
\begin{equation*}
\sum_{j=1}^{n}b_j-\sum_{i=1}^{m}a_{i}>-1.
\end{equation*}
then this series in \eqref{fox} is uniformly convergent for arbitrary $s\in \mathbb{C}$.

\textbf{Fractional calculus} is one of the fields of mathematical analysis which copes with exploration of fractional differential and integral operators which are non-local and work more accurative modelling ways than their appropriate integer-order versions. Fractional-order operators are more productive in modelling different disciplines like visco-elasticity \cite{viscoelasticity}, anomalous diffusion \cite{anomalous diffusion}, thermodynamics \cite{thermodynamics}, biophysics \cite{biophysics} and other areas.

We define some essential definitions related to fractional calculus that is going to be used throughout the paper.
 \begin{defn} {\cite{Kilbas,Samko et al.,Podlubny}}
 The Riemann-Liouville (R--L) fractional integral  of  order $\alpha \in \mathbb{C}$  with $\Re(\alpha)>0$ for a function   $g:[0,\infty)\to \mathbb{R}$ is defined by
 \begin{equation} \label{R-L integral}
 \prescript{}{}(I^{\alpha}_{a^{+}}g)(r)=\frac{1}{\Gamma(\alpha)}\int_a^r(r-s)^{\alpha-1}g(s)\,\mathrm{d}s \\,\quad \text{for} \quad r>a,
 \end{equation}
 \end{defn}
\begin{defn}\cite{Rainville}
The gamma function is defined as:
\begin{equation} \label{gamma}
\Gamma(\alpha)= \int_{0}^{\infty}\tau^{\alpha-1}e^{-\tau}\mathrm{d}\tau,  \alpha \in \mathbb{C} \quad \text{with} \quad \Re(\alpha) > 0.
\end{equation}
\end{defn}
 \begin{defn} \label{beta} \cite{Rainville}
 The beta function is defined as below:
 \begin{equation}
 B(c,d)=\int_{0}^{1}\tau^{c-1}(1-\tau)^{d-1}\mathrm{d}\tau, \quad  \text{for} \quad c,d \in \mathbb{C} \quad \text{with} \quad \Re(c),\Re(d)>0. 
 \end{equation}
 \end{defn}
 Furthermore, the beta function can be expressed with the aid of gamma functions \cite{Rainville} as below:
 \begin{equation*}
 B(c,d)=\frac{\Gamma(c)\Gamma(d)}{\Gamma(c+d)},\quad \text{for} \quad c,d \in \mathbb{C} \quad \text{with} \quad \Re(c),\Re(d)>0. 
 \end{equation*}
\begin{defn} {\cite{Kilbas,Samko et al.,Podlubny}}
 The R--L fractional derivative of order $\alpha \in \mathbb{C}$ with $\Re(\alpha)>0$ for a function  $g: [0, \infty) \to \mathbb{R}$ is defined by 
 \begin{equation} \label{R-L derivative}
 (\prescript{}{}D^{\alpha}_{a^{+}}g)(r)=\frac{\mathrm{d}^{n}}{\mathrm{d} r^{n}}(I^{n-\alpha}_{a^{+}}g)(r)\coloneqq\frac{1}{\Gamma(n-\alpha)}\frac{\mathrm{d}^{n}}{\mathrm{d} r^{n}}\int_a^r (r-s)^{n-\alpha-1}g(s)\mathrm{d}s, \quad \text{for} \quad n=\lfloor \Re(\alpha) \rfloor + 1 , r>a.
 \end{equation}
\end{defn}
\begin{defn}{\cite{Diethelm,Kilbas,Podlubny}}
 The Caputo fractional derivative  of  order $\alpha \in \mathbb{C}$ with $\Re(\alpha)>0$ for a function $g:[0,\infty)\to \mathbb{R}$ is defined by
 \begin{equation}
  (\prescript{C}{}D^{\alpha}_{a^{+}}g)(r)=\prescript{}{}I^{n-\alpha}_{a^{+}}\left( \frac{\mathrm{d}^n}{\mathrm{d} r^n}g\right)(r) \coloneqq\frac{1}{\Gamma(n-\alpha)}\int_a^r(r-s)^{n-\alpha-1}\frac{\mathrm{d}^n}{\mathrm{d}s^n}g(s)\,\mathrm{d}s,\quad \text{for} \quad n=\lfloor \Re(\alpha) \rfloor + 1 , r>a.
 \end{equation}
 In particular,
 \begin{equation*}
 \prescript{}{}I^{\alpha}_{a^{+}}\prescript{C}{}D^{\alpha}_{a^{+}}g(r)=g(r)-g(a) \quad \text{where} \quad 0<\alpha<1,\quad  r>a.
 \end{equation*} 
\end{defn}
\begin{defn} \cite{Diethelm,Kilbas}\label{caputo}
 The Caputo fractional derivative of order $\alpha \in (0,1)$ for a function $g:[0, \infty) \to \mathbb{R}$ can be written as 
 \begin{equation} \label{case0}
 (\prescript{C}{}{D^{\alpha}_{0^{+}}g})(r)= (\prescript{}{}{D^{\alpha}_{0^{+}}g})(r)-\frac{g(0)}{\Gamma(1-\alpha)}r^{-\alpha},\quad r>0.
 \end{equation}
\end{defn}

\textbf{Fractional differential equations} are a generalization of the classical ordinary and partial differential equations, in which the order of differentiation is permitted to be any real (or even complex) number, not only a natural number. FDEs are widely used to model mathematical problems in stability theory \cite{stability theory,ahmadova-mahmudov}, positive time-delay systems \cite{positive systems}, control theory \cite{control theory}, stochastic analysis \cite{Arzu}, electrical circuits \cite{electrical circuits,positive linear systems} and other areas.

FDEs containing not only one fractional derivative \cite{cell,Garra-Garrapa,Soubhia et.al.,Camargo et.al.,Oliveira et.al.,Huseynov-Mahmudov} but also more than one fractional derivative are intensively studied in many physical processes. Many authors demonstrate two essential mathematical ways to use this idea: multi-term equations \cite{Luchko-Gorenflo,Bazhlekova,Joel et al.,huseynov-mahmudov-1} and multi-order systems \cite{ismail-arzu,positive linear systems}.

 Multi-term FDEs have been studied due to their applications in modelling, and solved using various mathematical methods. Luchko and Gorenflo \cite{Luchko-Gorenflo} solved the following multi-term FDEs with constant coefficients and with the Caputo fractional derivatives  by using the method of operational calculus.

\begin{thm}\cite{Luchko-Gorenflo}
Let $\alpha>\alpha_1> \cdots > \alpha_n \geq 0$, $l_i-1 < \alpha_i\leq l_i$, $l_i \in \mathbb{N}_0=\mathbb{N} \cup \left\lbrace 0\right\rbrace$, $\mu_{i} \in \mathbb{R}, i=1, \cdots, n$. The initial value problem (IVP)
\begin{equation}\label{Cauchy problem}
\begin{cases*}
(\prescript{C}{}D^{\alpha}_{0_+}y)(r)-\sum\limits_{i=1}^{n}\mu_{i}(\prescript{C}{}D^{\alpha_i}_{0_+}y)(r)=g(r),\quad r>0 \\
y^{(k)}(0)=a_{k}\in \mathbb{R},\quad k=0, \cdots, l-1, \quad l-1 < \alpha \leq l,\quad l\in \mathbb{N},
\end{cases*}
\end{equation}
has a unique solution of the form
 \begin{equation*}
y(r)=y_{par}(r) + \sum_{k=0}^{l-1}a_k u_k(r), \quad r \geq 0,
\end{equation*}
here
\begin{equation*}
x_{par}(r)=\int_{0}^{r}s^{\alpha-1}E_{(\cdot),\alpha}(s)g(r-s)\mathrm{d}s,
\end{equation*}
is a particular solution of the IVP \eqref{Cauchy problem} with homogeneous initial condition, and the functions 
\begin{equation*}
u_k(r)=\frac{r^{k}}{k!}+\sum_{i=m_k+1}^{n}\mu_{i}r^{k+\alpha-\alpha_i}E_{(\cdot), k+1+\alpha-\alpha_{i}}(r), \quad k=0, \cdots, l-1,
\end{equation*}	
satisfying the following initial conditions
\begin{equation*}
u_{k}^{(m)}(0)=\delta_{km}=\begin{cases}
1, \quad k=m,\\ 0, \quad k \neq m, \quad \text{where} \quad k,m=0, \ldots, l-1.
\end{cases}
\end{equation*}
The function
\begin{equation}\label{Emulti}
E_{(\cdot),\beta}(r)=E_{(\alpha-\alpha_1,...,\alpha-\alpha_n),\beta}(\mu_1r^{\alpha-\alpha_1},...,\mu_n r^{\alpha-\alpha_n})
\end{equation}
is a particular case of the multivariate M--L function \eqref{multidimensional} and $m_k \in \mathbb{N}$ for $k=0, \ldots, l-1$ are defined for the condition 
\begin{equation*}
\begin{cases*}
l_{m_k} \geq k+1, \\
l_{m_{k+1}} \leq k.
\end{cases*}
\end{equation*}
In the particular case $l_i \leq k,i=0,...,l-1,$ we form $m_k\coloneqq 0$, and if $l_i \geq k+1, i=0,...,l-1$,then $m_k \coloneqq n$.
\end{thm}
In terms of numerical methods, Edwards et al. \cite{Edwards} and Diethelm et al. \cite{Kai-Luchko} have investigated the IVP for the general linear multi-term FDEs with constant coefficients. The authors in \cite{Kai-Luchko} have proposed a new algorithm for the numerical solution of the IVP \eqref{Cauchy problem}.

Furthermore, Bazhlekova \cite{Bazhlekova} have considered the following Caputo type fractional relaxation equations with multi-orders:
\begin{equation}
\begin{cases*}
(\prescript{C}{}D^{\alpha}_{0_+}y)(r)+\sum\limits_{j=1}^{l}\mu_{j}(\prescript{C}{}D^{\alpha_j}_{0_+}y)(r)+ \mu y(r)=g(r), r > 0,\\
y(0)=y_{0}\in \mathbb{R},   
\end{cases*}
\end{equation}
\text{where} $0 < \alpha_l < ... < \alpha_1 < \alpha \leq 1, \mu,\mu_j > 0, j=1,...,l, l\in \mathbb{N}_0$.
Applying Laplace transformation, the fundamental solutions of the IVP are studied in \cite{Bazhlekova}.

In the same vein as above articles, we propose the exact analytical representation of solutions of Cauchy problem for a FDE with three independent fractional orders by introducing a newly defined trivariate M--L function 
\begin{equation} \label{1.1}
\begin{cases}
(\prescript{C}{}D^{\alpha}_{0+}y)(r)-\lambda_{3}(\prescript{C}{}D^{\beta}_{0+}y)(r)-\lambda_{2}(\prescript{C}{}D^{\gamma}_{0+}y)(r)-\lambda_{1}y(r)= g(r), r>0, \\ y(0)=y_0 \in \mathbb{R},
\end{cases}
\end{equation}
where $(\prescript{C}{}D^{\alpha}_{0^{+}}y)(\cdot), (\prescript{C}{}D^{\beta}_{0^{+}}y)(\cdot)$ and $(\prescript{C}{}D^{\gamma}_{0^{+}}y)(\cdot) $ are the Caputo type fractional differentiation operators of orders $1 \geq \alpha>\beta>\gamma >0 $ , $\lambda_i \in \mathbb{R},i=1,2,3$ denote constants and  $g \in  C([0,T], \mathbb{R})$.

In terms of Laplace integral transform method, Kilbas et al. \cite{Kilbas} have considered the Cauchy problem for \eqref{1.1} by using generalized Wright functions, in both homogeneous and inhomogeneous cases. It is necessary to note that our results by means of new trivariate M--L functions coincide with the results by means of Fox-Wright functions in \cite{Kilbas}.

The structure of this paper contains important improvement in the theory of special functions and multi-term FDEs and is outlined as below.
In Sect.\ref{sec:trivariate}, we establish a new trivariate Mittag-Leffler
function as a natural extension of classical M--L functions. We establish different properties of these function, also complex integral representation and Laplace integral transform and appropriate convolution results. \\
In Sect.\ref{sec:3}, firstly we consider $n$-th order derivative and integration. Then we investigate fractional order integral and derivatives of a newly defined M--L type function in Rieamann-Liouville and Caputo senses. Sect.\ref{sec:multi} is devoted to presenting the exact analytical solution by means of triple infinite series to the homogeneous linear multi-term FDE. Furthermore, we describe the exact solutions of the inhomogeneous linear FDE via a method of variation of constants formula via classical ideas. Sect.\ref{sec:ex} is related to illustrating an example to guarantee an ability of the given  analytical solutions of \eqref{1.1}. In Sect.\ref{sec:concl}, we discuss our main contributions of this paper and future research work.
 
\section{The new trivariate Mittag--Leffler function} \label{sec:trivariate}
\subsection{Introducing the definition}
\begin{defn} \label{def2}
	Let $\alpha, \beta, \gamma, \delta, \eta \in \mathbb{C}$ with $\Re(\alpha)>0$, $\Re(\beta)>0$, and $\Re(\gamma)>0$. We propose the following trivariate Mittag--Leffler (M--L) function:
	\begin{equation}\label{tML4}
	E_{\alpha,\beta,\gamma, \delta}^{\eta}(u,v,w)=\sum_{l=0}^{\infty}\sum_{p=0}^{\infty}\sum_{k=0}^{\infty}\frac{(\eta)_{l+p+k}}{\Gamma(l\alpha+p\beta+k\gamma+\delta)}\frac{u^lv^pw^k}{l!p!k!},\quad u,v,w \in \mathbb{C},
	\end{equation}
	where the numerator is a Pochhammer symbol which satisfies the following identity \cite{whittaker-watson}: 
	\begin{equation}\label{pochhammer}
	(\eta)_{l+p+k}=(\eta)_{l+p}(\eta+l+p)_{k}=(\eta)_{l}(\eta+l)_{p}(\eta+l+p)_{k}.
	\end{equation}
\end{defn}

In the special cases, whenever $w=0$ and $v=w=0$, trivariate Mittag-Leffler function \eqref{tML4}  reduces to  bivariate Miitag-Leffler \eqref{bML4} and three-parameter Mittag-Leffler (Prabhakar's) \eqref{Prabf} functions, respectively.

When we substitute $u=\lambda_{1}r^{\alpha}$, $v=\lambda_{2}r^{\beta}$, and $w=\lambda_{3}r^{\gamma}$ in \eqref{tML4}, then we can deduce the new trivariate Mittag-Leffler type function as follows:
\begin{equation}\label{trivtype}
r^{\delta-1}E_{\alpha,\beta,\gamma, \delta}^{\eta}(\lambda_{1}r^{\alpha},\lambda_{2}r^{\beta}, \lambda_{3}r^{\gamma})=\sum_{l=0}^{\infty}\sum_{p=0}^{\infty}\sum_{k=0}^{\infty}\frac{(\eta)_{l+p+k}}{\Gamma(l\alpha+p\beta+k\gamma+\delta)}\frac{\lambda_{1}^l\lambda_{2}^p\lambda_{3}^k}{l!p!k!}r^{l\alpha+p\beta+k\gamma+\delta-1}.
\end{equation}
It should be pointed out that this series in \eqref{tML4} converges absolutely and locally uniformly for $\alpha,\beta,\gamma \in \mathbb{C}$ with $\Re(\alpha)>0$, $\Re(\beta)>0$, and $\Re(\gamma)>0$. It can be easily proved by making use of the technique for convergence of the generalized Lauricella series in three variables which is proposed by Srivastava and Daoust \cite{ Srivastava-Daoust, Srivastava}.
\begin{lem}\label{lem:texp}
	If whole parameters are equal to $1$, then we get the triple exponential function:
	\begin{equation*}
	E_{1,1,1,1}^{1}(u,v,w)=\exp(u)\exp(v)\exp(w)=\exp(u+v+w), \quad u,v,w \in \mathbb{C}.
	\end{equation*}
\end{lem}
\begin{proof}
Applying Definition \ref{def2},
	\begin{align*}
	E_{1,1,1,1}^{1}(u,v,w)&=\sum_{l=0}^{\infty}\sum_{p=0}^{\infty}\sum_{k=0}^{\infty}\frac{(1)_{l+p+k}}{\Gamma(l+p+k+1)}\frac{u^lv^pw^k}{l!p!k!}\\
	&=\sum_{l=0}^{\infty}\sum_{p=0}^{\infty}\sum_{k=0}^{\infty}\frac{\Gamma(l+p+k+1)}{\Gamma(1)\Gamma(l+p+k+1)}\frac{u^lv^pw^k}{l!p!k!}\\
	&=\sum_{l=0}^{\infty}\frac{u^l}{l!}\sum_{p=0}^{\infty}\frac{v^p}{p!}\sum_{k=0}^{\infty}\frac{w^k}{k!}=\exp(u)\exp(v)\exp(w)=\exp(u+v+w).
	\end{align*}
\end{proof}
\begin{remark}
	Lemma \ref{lem:texp} is the natural extension of M--L functions with one or two variables that
	\begin{equation*}
	E_{1}(u)=\exp(u) \quad \text{and} \quad E_{1,1,1}^{1}(u,v)=\exp(u)\exp(v)=\exp(u+v).
	\end{equation*}
\end{remark}
For simplicity, we denote $E_{\alpha,\beta,\gamma, \delta}^{1}(\lambda_{1}r^{\alpha},\lambda_{2}r^{\beta},\lambda_{3}r^{\gamma})\coloneqq E_{\alpha,\beta,\gamma, \delta}(\lambda_{1}r^{\alpha},\lambda_{2}r^{\beta},\lambda_{3}r^{\gamma})$ in our results for this paper.

In Figure \ref{fig:plot1}, we present the bivariate M-L function when $\eta=1$ and $w=0$ in \eqref{tML4}.
\begin{figure}[H]
	\centering
	\begin{subfigure}[b]{0.36\linewidth}
		\includegraphics[width=\linewidth]{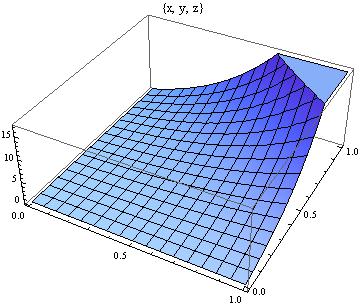}
		\caption{$\alpha=0.8$, $\beta=0.6$, $\delta=0.2$}
	\end{subfigure}
	\begin{subfigure}[b]{0.36\linewidth}
		\includegraphics[width=\linewidth]{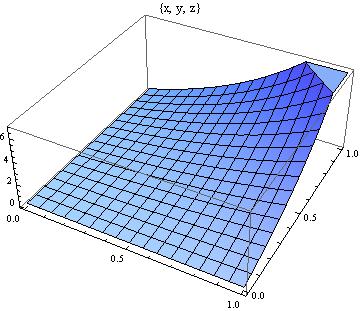}
		\caption{$\alpha=0.9$, $\beta=0.8$, $\delta=0.6$}
	\end{subfigure}
\caption{The plots of the bivariate M-L function of
$E_{\alpha,\beta,\delta}(u,v)$ with $\eta=1$ and varying values of $\alpha,\beta, \delta$}
\label{fig:plot1}
\end{figure}
\vspace{0.5cm}
In Figure \ref{fig:plot2}, we present the three-parameter M-L function when $\eta=1$ and $v=w=0$ in \eqref{tML4}.
\begin{figure}[H]
	\centering
	\begin{subfigure}[b]{0.36\linewidth}
		\includegraphics[width=\linewidth]{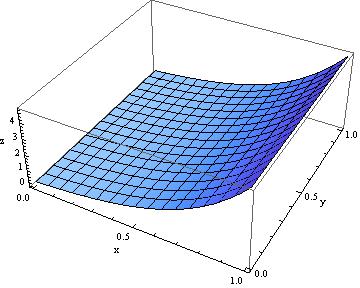}
		\caption{$\alpha=0.6$, $\delta=0.1$}
	\end{subfigure}
	\begin{subfigure}[b]{0.36\linewidth}
		\includegraphics[width=\linewidth]{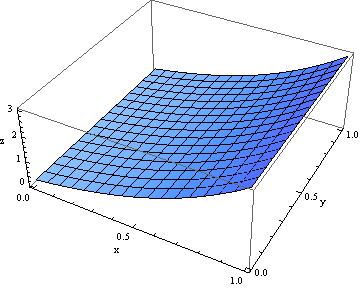}
		\caption{$\alpha=0.9$, $\delta=0.3$}
	\end{subfigure}
	\caption{The plots of the three-parameter M-L function of
		$E_{\alpha,\delta}(u)$ with $\eta=1$ and varying values of $\alpha, \delta$}
	\label{fig:plot2}
\end{figure}
\vspace{0.5cm}
In Figure \ref{fig:plot}, we describe univariate version of the trivariate M--L function \eqref{trivtype} with $\delta=\eta=1$ and different parameters $\alpha,\beta, \gamma$.
\begin{figure}[H]
	\centering
	\begin{subfigure}[b]{0.36\linewidth}
		\includegraphics[width=\linewidth]{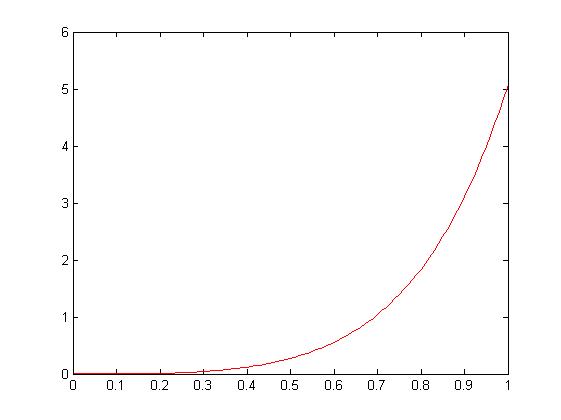}
		\caption{$E_{1,1,1,1}^{1}(r,r,r)=\exp(3r)$}
	\end{subfigure}
	\begin{subfigure}[b]{0.36\linewidth}
		\includegraphics[width=\linewidth]{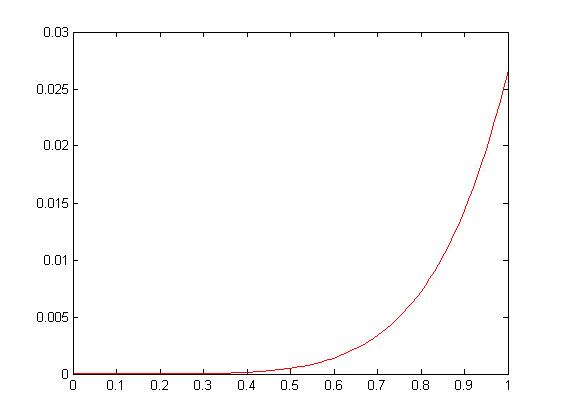}
		\caption{$\alpha=1.5$, $\beta=1.5$, $\gamma=2.5$}
	\end{subfigure}

	\begin{subfigure}[b]{0.36\linewidth}
		\includegraphics[width=\linewidth]{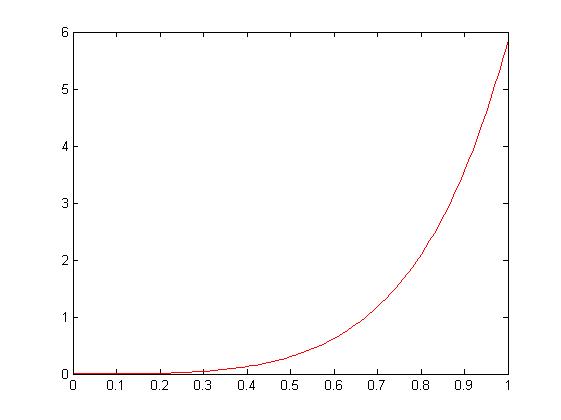}
		\caption{$\alpha=0.75$, $\beta=1$, $\gamma=1.25$}
	\end{subfigure}
	\begin{subfigure}[b]{0.36\linewidth}
		\includegraphics[width=\linewidth]{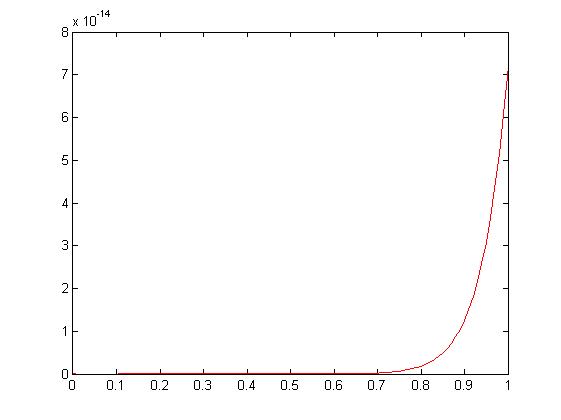}
		\caption{$\alpha=8$, $\beta=6$, $\gamma=4$}
	\end{subfigure}
	\caption{The plots of the univariate version of $E_{\alpha,\beta,\gamma,1}^{1}(r^{\alpha},r^{\beta},r^{\gamma})$ with $\delta=\eta=1$ and varying $\alpha,\beta, \gamma$}
	\label{fig:plot}
\end{figure}
\vspace{0.5cm}

The following Figure \ref{fig:fig2} shows the comparison of new trivariate, bivariate and M--L functions with two and three parameters.
\begin{figure}[H]
	\centering
	\includegraphics[width=0.6\linewidth]{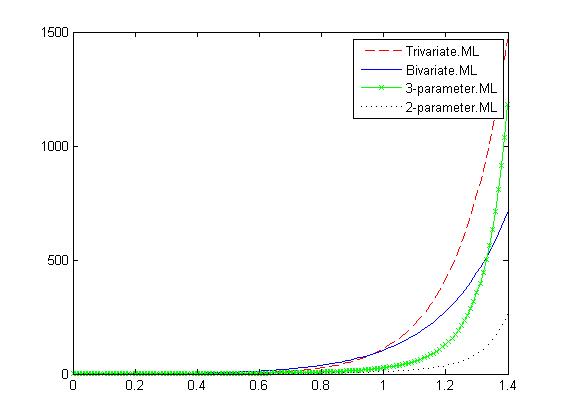}
	\caption{The comparison of various types of M--L functions}
	\label{fig:fig2}
\end{figure}
\vspace{0.5cm}
Here we provide the values of each parameter through the following table which are used in Figure \ref{fig:fig2} to compare final results for each functions.
\begin{table}[H]
	\centering
	\begin{tabular}{c c c c c} 
		\hline 
		$ $ & $E^{\eta}_{\alpha,\beta,\gamma,\delta}$ & $E^{\eta}_{\alpha,\beta,\gamma}$ & $E^{\eta}_{\alpha,\beta}$ & $E^{\eta}_{\alpha,\beta}$ \\ [1ex] 
		\hline
     	$\alpha$ &	$0.25$ & $0.25$ & $0.25$ & $0.25$ \\[0.5ex]
     	\hline
    	$\beta$ &	$0.75$ & $0.75$ & $0.75$ & $0.75$ \\[0.5ex]
    	\hline
    	$\gamma$ &	$1.5$ & $1.5$ & $-$ & $-$ \\[0.5ex]
    	\hline
    	$\delta$ &	$1.5$ & $-$ & $-$ & $-$ \\[0.5ex]
    	\hline
    	$\eta$ &	$1$ & $1$ & $1.5$ & $1$ \\[0.5ex] 
    	\hline
	\end{tabular}
	\caption{The values of parameters for each functions}
\end{table}
\vspace{0.1cm}
\subsection{Main results and relationships} 
\begin{thm}\label{complex}
For  $\alpha, \beta, \gamma, \delta, \eta \in \mathbb{C}$ with 	$\Re(\alpha)>0$, $\Re(\beta)>0$, and $\Re(\gamma)>0$, the trivariate M--L function \eqref{tML4} has the following integral representation on the complex plane:
\begin{equation}
E_{\alpha,\beta,\gamma, \delta}^{\eta}(u,v,w)=\frac{1}{2\pi i}\int\limits_{H}\frac{e^{\tau}\tau^{-\delta}}{(1-u\tau^{-\alpha}-v\tau^{-\beta}-w\tau^{-\gamma})^{\eta}}\mathrm{d}\tau,
\end{equation}
where $H$ is the Hankel contour.
\end{thm}
\begin{proof}
Using up the well-known Hankel formula for reciprocal of the gamma function \cite{whittaker-watson}:
	\begin{equation*}
	\frac{1}{\Gamma(s)}=\frac{1}{2\pi i}\int\limits_{H} \tau^{-s}e^{\tau}\mathrm{d}\tau, \quad s \in \mathbb{C}.
	\end{equation*}
	
	Thus for $E_{\alpha,\beta,\gamma, \delta}^{\eta}(u,v,w)$, we have
	\begin{align*}
	E_{\alpha,\beta,\gamma, \delta}^{\eta}(u,v,w)&=\sum_{l=0}^{\infty}\sum_{p=0}^{\infty}\sum_{k=0}^{\infty}\frac{(\eta)_{l+p+k}}{\Gamma(l\alpha+p\beta+k\gamma+\delta)}\frac{u^lv^pw^k}{l!p!k!}\\
	&=\frac{1}{2\pi i}\sum_{l=0}^{\infty}\sum_{p=0}^{\infty}\sum_{k=0}^{\infty}\frac{(\eta)_{l+p+k}u^lv^pw^k}{l!p!k!}\int\limits_{H} \tau^{-l\alpha-p\beta-k\gamma-\delta}e^{\tau}\mathrm{d}\tau\\
	&=\frac{1}{2\pi i}\int\limits_{H} e^{\tau}\tau^{-\delta}\sum_{l=0}^{\infty}\sum_{p=0}^{\infty}\sum_{k=0}^{\infty}\frac{(\eta)_{l+p+k}}{l!p!k!}\left( \frac{u}{\tau^{\alpha}}\right) ^{l}\left( \frac{v}{\tau^{\beta}}\right)^{p}\left( \frac{w}{\tau^{\gamma}}\right)^{k}\mathrm{d}\tau.
	\end{align*}
	
	We evaluate the triple sum in this integral above by using  \eqref{pochhammer} and the following identity
	\begin{equation}\label{sum}
	\sum_{m=0}^{\infty}\frac{(\rho)_m}{m!}s^{m}=(1-s)^{-\rho}, \rho \in \mathbb{C}
	\end{equation}
	as  follows:
	\allowdisplaybreaks
   \begin{align*}
    &\sum_{l=0}^{\infty}\sum_{p=0}^{\infty}\sum_{k=0}^{\infty}\frac{(\eta)_{l+p+k}}{l!p!k!}\left( \frac{u}{\tau^{\alpha}}\right)^{l}\left( \frac{v}{\tau^{\beta}}\right)^{p}\left( \frac{w}{\tau^{\gamma}}\right)^{k}\\
    =&\sum_{l=0}^{\infty}\sum_{p=0}^{\infty}\sum_{k=0}^{\infty}\frac{(\eta)_{l}(\eta+l)_{p}(\eta+l+p)_{k}}{l!p!k!}\left( \frac{u}{\tau^{\alpha}}\right) ^{l}\left(\frac{v}{\tau^{\beta}}\right)^{p}\left( \frac{w}{\tau^{\gamma}}\right)^{k}\\
	=&\sum_{l=0}^{\infty}\frac{(\eta)_{l}}{l!}\left( \frac{u}{\tau^{\alpha}}\right) ^{l}\sum_{p=0}^{\infty}\sum_{k=0}^{\infty}\frac{(\eta+l+p)_{k}(\eta+l)_{p}}{p!k!}\left( \frac{v}{\tau^{\beta}}\right)^{p}\left( \frac{w}{\tau^{\gamma}}\right)^{k}\\
	=&\sum_{l=0}^{\infty}\frac{(\eta)_{l}}{l!}\left( \frac{u}{\tau^{\alpha}}\right) ^{l}\sum_{p=0}^{\infty}\sum_{k=0}^{\infty}\frac{(\eta+l)_{p+k}}{p!k!}\left( \frac{v}{\tau^{\beta}}\right)^{p}\left( \frac{w}{\tau^{\gamma}}\right)^{k}\\
	=&\sum_{l=0}^{\infty}\frac{(\eta)_{l}}{l!}\left( \frac{u}{\tau^{\alpha}}\right)^{l} \Big(1-\frac{v}{\tau^{\beta}}-\frac{w}{\tau^{\gamma}}\Big)^{-\eta-l}\\
	=&\Big(1-\frac{v}{\tau^{\beta}}-\frac{w}{\tau^{\gamma}}\Big)^{-\eta}\sum_{l=0}^{\infty}\frac{(\eta)_{l}}{l!}\Big(\frac{u}{\tau^{\alpha}(1-\frac{v}{\tau^{\beta}}-\frac{w}{\tau^{\gamma}})}\Big)^{l}\\
	=&\Big(1-\frac{v}{\tau^{\beta}}-\frac{w}{\tau^{\gamma}}\Big)^{-\eta}\Big(1-\frac{u}{\tau^{\alpha}(1-\frac{v}{\tau^{\beta}}-\frac{w}{\tau^{\gamma}})}\Big)^{-\eta}\\
	=&\Big(1-\frac{u}{\tau^{\alpha}}-\frac{v}{\tau^{\beta}}-\frac{w}{\tau^{\gamma}}\Big)^{-\eta}=\frac{1}{(1-\frac{u}{\tau^{\alpha}}-\frac{v}{\tau^{\beta}}-\frac{w}{\tau^{\gamma}})^{\eta}}.
	\end{align*}
Plugging this to the integral formula for $E_{\alpha,\beta,\gamma, \delta}^{\eta}(u,v,w)$ attained above, we get the desired result.
\end{proof}

\begin{coroll}
	Let $\alpha, \beta, \gamma, \delta, \eta \in \mathbb{C}$, with $\Re(\alpha)>0$, $\Re(\beta)>0$, and $\Re(\gamma)>0$, $\lambda_{1},\lambda_{2}, \lambda_{3} \in \mathbb{C}$  and $r \in \mathbb{C}$. The complex integral representation for the univariate version \eqref{trivtype} is defined by:
	\begin{equation*}
	r^{\delta -1}E_{\alpha,\beta,\gamma, \delta}^{\eta}(\lambda_{1}r^{\alpha},\lambda_{2}r^{\beta},\lambda_{3}r^{\gamma})=\frac{1}{2\pi i} \int\limits_{H}\frac{e^{rs}s^{-\delta}}{(1-\lambda_{1}s^{-\alpha}-\lambda_{2}s^{-\beta}-\lambda_{3}s^{-\gamma})^{\eta}}\mathrm{d}s.
	\end{equation*}
\end{coroll}
\begin{proof}
	Applying Theorem \ref{complex}, we make use of substitution $u=\lambda_{1}r^{\alpha}$, $v=\lambda_{2}r^{\beta}$ and $w=\lambda_{3}r^{\gamma}$, we obtain:
	\begin{equation*}
	E_{\alpha,\beta,\gamma, \delta}^{\eta}(\lambda_{1}r^{\alpha},\lambda_{2}r^{\beta},\lambda_{3}r^{\gamma})=\frac{1}{2\pi i}\int\limits_{H}\frac{e^{\tau}\tau^{-\delta}}{(1-\lambda_{1}\left( \frac{r}{\tau}\right) ^{\alpha}-\lambda_{2}\left( \frac{r}{\tau}\right)^{\beta}-\lambda_{3}\left( \frac{r}{\tau}\right)^{\gamma})^{\eta}}\mathrm{d}\tau,
	\end{equation*}
	and
	\begin{equation*}
	r^{\delta -1}E_{\alpha,\beta,\gamma, \delta}^{\eta}(\lambda_{1}r^{\alpha},\lambda_{2}r^{\beta},\lambda_{3}r^{\gamma})=\frac{1}{2\pi i}\int\limits_{H}\frac{e^{\tau}\left( \frac{r}{\tau}\right) ^{\delta}}{(1-\lambda_{1}\left( \frac{r}{\tau}\right) ^{\alpha}-\lambda_{2}\left( \frac{r}{\tau}\right)^{\beta}-\lambda_{3}\left( \frac{r}{\tau}\right)^{\gamma})^{\eta}}\frac{1}{r}\mathrm{d}\tau,
	\end{equation*}
thus, substitute $s=\frac{\tau}{r}$ to get the stated result:
	\begin{equation*}
	r^{\delta -1}E_{\alpha,\beta,\gamma, \delta}^{\eta}(\lambda_{1}r^{\alpha},\lambda_{2}r^{\beta},\lambda_{3}r^{\gamma})=\frac{1}{2\pi i}\int\limits_{H}\frac{e^{rs}s^{-\delta}}{(1-\lambda_{1}s^{-\alpha}-\lambda_{2}s^{-\beta}-\lambda_{3}s^{-\gamma})^{\eta}}\mathrm{d}s.
	\end{equation*}
	
\end{proof}
The next results concern the Laplace integral transform of univariate formula for trivariate M--L type function \eqref{trivtype}.

\begin{thm}\label{Lap}
For $\lambda_{i} \in \mathbb{C}, i =1,2,3$, $\alpha, \beta, \gamma, \delta, \eta \in \mathbb{C}$ with $\Re(\alpha)>0$, $\Re(\beta)>0$, $\Re(\gamma)>0$, and $\Re(\delta)>0$, the following holds:
\begin{equation*}
\mathcal{L}\left\lbrace r^{\delta-1}E_{\alpha,\beta,\gamma, \delta}^{\eta}(\lambda_{1}r^{\alpha},\lambda_{2}r^{\beta},\lambda_{3}r^{\gamma})\right\rbrace (s)=\frac{1}{s^{\delta}}\left( 1-\frac{\lambda_{1}}{s^{\alpha}}-\frac{\lambda_{2}}{s^{\beta}}-\frac{\lambda_{3}}{s^{\gamma}}\right)^{-\eta}, \quad \Re(s)>0.
\end{equation*}
\end{thm}
\begin{proof}
Since the triple series is locally and uniformly convergent, we can integrate it term by term. The Laplace integral transform of a power function is defined by
	\begin{equation*}
	\mathcal{L}\left\lbrace \frac{r^{l-1}}{\Gamma(l)}\right\rbrace (s)=\frac{1}{s^{l}}, \quad \Re(l)>-1.
	\end{equation*}
	Therefore, by using \eqref{sum} for the trivariate Mittag-Leffler type function we have:
	\allowdisplaybreaks
	\begin{align*}
	\mathcal{L}\left\lbrace r^{\delta-1}E_{\alpha,\beta,\gamma, \delta}^{\eta}(\lambda_{1}r^{\alpha},\lambda_{2}r^{\beta},\lambda_{3}r^{\gamma})\right\rbrace (s)&=\sum_{l=0}^{\infty}\sum_{p=0}^{\infty}\sum_{k=0}^{\infty}\frac{(\eta)_{l+p+k}}{\Gamma(l\alpha+p\beta+k\gamma+\delta)}\frac{\lambda_{1}^l\lambda_{2}^p\lambda_{3}^k}{l!p!k!} \mathcal{L}\left\lbrace r^{l\alpha+p\beta+k\gamma+\delta-1}\right\rbrace(s)\\
	&=\frac{1}{s^{\delta}}\sum_{l=0}^{\infty}\sum_{p=0}^{\infty}\sum_{k=0}^{\infty}\frac{(\eta)_{l+p+k}}{l!p!k!}\left(\frac{\lambda_{1}}{s^{\alpha}} \right)^{l}\left(\frac{\lambda_{2}}{s^{\beta}} \right)^{p}\left(\frac{\lambda_{3}}{s^{\gamma}} \right)^{k}\\
	&=\frac{1}{s^{\delta}}\sum_{l=0}^{\infty}\sum_{p=0}^{\infty}\frac{(\eta)_{l+p}}{l!p!}\left(\frac{\lambda_{1}}{s^{\alpha}} \right)^{l}\left(\frac{\lambda_{2}}{s^{\beta}} \right)^{p}\sum_{k=0}^{\infty}\frac{(\eta+l+p)_{k}}{k!}\left(\frac{\lambda_{3}}{s^{\gamma}} \right)^{k}\\
	&\hspace{-5.5cm}=\frac{1}{s^{\delta}}\left( 1-\frac{\lambda_{3}}{s^{\gamma}}\right)^{-\eta} \sum_{l=0}^{\infty}\frac{(\eta)_{l}}{l!}\left(\frac{\lambda_{1}}{s^{\alpha}} \right)^{l}\left(1- \frac{\lambda_{3}}{s^{\gamma}}\right)^{-l}\sum_{p=0}^{\infty}\frac{(\eta+l)_{p}}{p!}\left(\frac{\lambda_{2}}{s^{\beta}} \right)^{p} 
	\left( 1-\frac{\lambda_{3}}{s^{\gamma}}\right)^{-p}\\
	&\hspace{-5.5cm}=\frac{1}{s^{\delta}}\left( 1-\frac{\lambda_{3}}{s^{\gamma}}\right)^{-\eta}\left(1-\frac{\lambda_{2}}{s^{\beta}}\left(1-\frac{\lambda_{3}}{s^{\gamma}}\right)^{-1} \right)^{-\eta}\sum_{l=0}^{\infty}\frac{(\eta)_{l}}{l!}\left(\frac{\lambda_{1}}{s^{\alpha}}\right)^{l}\left(1- \frac{\lambda_{3}}{s^{\gamma}}\right)^{-l}\left(  1-\frac{\lambda_{2}}{s^{\beta}}\left(1- \frac{\lambda_{3}}{s^{\gamma}}\right)^{-1} \right)^{-l} \\
	&=\frac{1}{s^{\delta}}\left( 1-\frac{\lambda_{2}}{s^{\beta}}-\frac{\lambda_{3}}{s^{\gamma}}\right)^{-\eta}\sum_{l=0}^{\infty}\frac{(\eta)_{l}}{l!}\left( \frac{\lambda_{1}}{s^{\alpha}}\left( 1-\frac{\lambda_{2}}{s^{\beta}}-\frac{\lambda_{3}}{s^{\gamma}}\right)^{-1}\right)^{l}\\
	&=\frac{1}{s^{\delta}}\left( 1-\frac{\lambda_{2}}{s^{\beta}}-\frac{\lambda_{3}}{s^{\gamma}}\right)^{-\eta}\left( 1-\frac{\lambda_{1}}{s^{\alpha}}\left( 1-\frac{\lambda_{2}}{s^{\beta}}-\frac{\lambda_{3}}{s^{\gamma}}\right)^{-1}\right)^{-\eta}\\
	&=\frac{1}{s^{\delta}}\left( 1-\frac{\lambda_{1}}{s^{\alpha}}-\frac{\lambda_{2}}{s^{\beta}}-\frac{\lambda_{3}}{s^{\gamma}}\right)^{-\eta}.
	\end{align*}
Note that we have need of extra conditions on $s$:	 
\begin{equation*}
\left\lvert \frac{\lambda_{3}}{s^{\gamma}}\right \rvert <1, \quad \left \lvert\frac{\lambda_{2}}{s^{\beta}}\left( 1-\frac{\lambda_{3}}{s^{\gamma}}\right)^{-1}\right\rvert <1  \quad \text{and} \quad  \left \lvert\frac{\lambda_{1}}{s^{\alpha}}\left( 1-\frac{\lambda_{2}}{s^{\beta}}-\frac{\lambda_{3}}{s^{\gamma}}\right)^{-1}\right \rvert <1,
\end{equation*}
for proper convergence of the series. However, these conditions can be reduced according to the analytic continuation. Therefore, this gives the desired result for arbitrary $s \in \mathbb{C}$ whenever $\Re(s)>0$.
\end{proof}
Next we prove a result of convolution on trivariate Mittag-Leffler type functions which is related to above theorem directly.
\begin{thm}
Let $\lambda_{i} \in \mathbb{C}, i=1,2,3$, $\alpha, \beta, \gamma, \delta_{1}, \delta_{2},\eta_{1},\eta_{2} \in \mathbb{C}$ with $\Re(\alpha)>0$, $\Re(\beta)>0$, $\Re(\gamma)>0$ and $\Re(\delta_{j})>0, j=1,2$. Then the next result yields:
	\begin{align}\nonumber
	\left( r^{\delta_{1}-1} E_{\alpha,\beta,\gamma, \delta_1}^{\eta_1}(\lambda_{1}r^{\alpha},\lambda_{2}r^{\beta}, \lambda_{3}r^{\gamma})\right) &\ast \left( r^{\delta_{2}-1} E_{\alpha,\beta,\gamma, \delta_2}^{\eta_2}(\lambda_{1}r^{\alpha},\lambda_{2}r^{\beta}, \lambda_{3}r^{\gamma})\right)\\
	&=r^{\delta_{1}+\delta_{2}-1}E_{\alpha,\beta,\gamma, \delta_1+\delta_{2}}^{\eta_1+\eta_{2}}(\lambda_{1}r^{\alpha},\lambda_{2}r^{\beta}, \lambda_{3}r^{\gamma}).
	\end{align}
\end{thm}
\begin{proof}
	By using the theorem of convolution for the Laplace transformation and Theorem \ref{Lap}, we get
\begin{align*}
&\mathcal{L}\Biggl\{\left( r^{\delta_{1}-1} E_{\alpha,\beta,\gamma, \delta_1}^{\eta_1}(\lambda_{1}r^{\alpha},\lambda_{2}r^{\beta}, \lambda_{3}r^{\gamma})\right) \ast \left( r^{\delta_{2}-1} E_{\alpha,\beta,\gamma, \delta_2}^{\eta_2}(\lambda_{1}r^{\alpha},\lambda_{2}r^{\beta}, \lambda_{3}r^{\gamma})\right) \Biggr\}(s)\\
=&\mathcal{L}\left\lbrace r^{\delta_{1}-1} E_{\alpha,\beta,\gamma, \delta_1}^{\eta_1}(\lambda_{1}r^{\alpha},\lambda_{2}r^{\beta}, \lambda_{3}r^{\gamma}) \right\rbrace (s) \mathcal{L}\left\lbrace r^{\delta_{2}-1} E_{\alpha,\beta,\gamma, \delta_2}^{\eta_2}(\lambda_{1}r^{\alpha},\lambda_{2}r^{\beta}, \lambda_{3}r^{\gamma})\right\rbrace(s) \\
=&\frac{1}{s^{\delta_1}}\left( 1-\frac{\lambda_{1}}{s^{\alpha}}-\frac{\lambda_{2}}{s^{\beta}}-\frac{\lambda_{3}}{s^{\gamma}}\right)^{-\eta_{1}}\frac{1}{s^{\delta_2}}\left(1- \frac{\lambda_{1}}{s^{\alpha}}-\frac{\lambda_{2}}{s^{\beta}}-\frac{\lambda_{3}}{s^{\gamma}}\right)^{-\eta_2}\\
=&\frac{1}{s^{\delta_1+\delta_2}}\left( 1-\frac{\lambda_{1}}{s^{\alpha}}-\frac{\lambda_{2}}{s^{\beta}}-\frac{\lambda_{3}}{s^{\gamma}}\right)^{-(\eta_1+\eta_2)}\\
=&\mathcal{L}\left\lbrace r^{\delta_{1}+\delta_{2}-1}E_{\alpha,\beta,\gamma, \delta_1+\delta_{2}}^{\eta_1+\eta_{2}}(\lambda_{1}r^{\alpha},\lambda_{2}r^{\beta}, \lambda_{3}r^{\gamma}) \right\rbrace (s).
\end{align*}
Taking inverse Laplace transform both sides to the above expression, we acquire the desired result.
\end{proof}
\section{Fractional calculus of trivariate M--L function} \label{sec:3}
 In this section firstly, we investigate $n$-th order derivative and integration of a newly defined trivariate Mittag-Leffler type function. Next using these results we will investigate fractional derivative and fractional integral of a trivarivate M--L function in R--L and Caputo senses.
 
 \begin{thm}
 	Let $\alpha,\beta,\gamma,\delta,\eta,\lambda_i \in \mathbb{C}$ with $\Re(\alpha),\Re(\beta)$,and $\Re(\gamma) > 0,i=1,2,3$. Then for arbitrary $n \in \mathbb{N}$, the following formula holds true:
 \begin{equation}\label{ordinary derivative}
 \left(\frac{d}{dr}\right)^n\left[ r^{\delta-1} E_{\alpha,\beta,\gamma,\delta}^{\eta}(\lambda_1 r^{\alpha},\lambda_2 r^{\beta},\lambda_3 r^\gamma)\right] =r^{\delta-n-1} E_{\alpha,\beta,\gamma,\delta-n}^{\eta}(\lambda_1r^\alpha,\lambda_2r^\beta,\lambda_3r^\gamma).
 \end{equation}
 \end{thm}
\begin{proof}
By using \eqref{trivtype} and differentiating term by term under the summation signs, we acquire that
\begin{align*}
&\left(\frac{d}{dr}\right)^n\left[ r^{\delta-1} E_{\alpha,\beta,\gamma,\delta}^{\eta}(\lambda_1r^{\alpha},\lambda_2r^{\beta},\lambda_3r^\gamma)\right]\\ &=\sum_{l=0}^{\infty}\sum_{p=0}^{\infty}\sum_{k=0}^{\infty}\frac{(\eta)_{l+p+k}}{\Gamma(l\alpha+p\beta+k\gamma+\delta)}\frac{\lambda_{1}^l\lambda_{2}^p\lambda_{3}^k}{l!p!k!} \left( \frac{d}{dr}\right)^n [r^{l\alpha+p\beta+k\gamma+\delta-1}]\\
&=\sum_{l=0}^{\infty}\sum_{p=0}^{\infty}\sum_{k=0}^{\infty}\frac{(\eta)_{l+p+k}}{\Gamma(l\alpha+p\beta+k\gamma+\delta)}\frac{\lambda_{1}^l\lambda_{2}^p\lambda_{3}^k}{l!p!k!} \frac{\Gamma(l\alpha+p\beta+k\gamma+\delta)}{\Gamma(l\alpha+p\beta+k\gamma+\delta-n)} [r^{l\alpha+p\beta+k\gamma+\delta-n-1}]\\
&=\sum_{l=0}^{\infty}\sum_{p=0}^{\infty}\sum_{k=0}^{\infty}\frac{(\eta)_{l+p+k}}{\Gamma(l\alpha+p\beta+k\gamma+\delta-n)}\frac{\lambda_{1}^l\lambda_{2}^p\lambda_{3}^k}{l!p!k!} r^{l\alpha+p\beta+k\gamma+\delta-n-1}\\
&=r^{\delta-n-1} E_{\alpha,\beta,\gamma,\delta-n}^{\eta}(\lambda_1r^\alpha,\lambda_2r^\beta,\lambda_3r^\gamma), 
\end{align*}
which proves \eqref{ordinary derivative}. 
\end{proof}
\begin{coroll}
Let $\alpha,\beta,\gamma,\delta,\eta,\lambda_i \in \mathbb{C}$ with $\Re(\alpha),\Re(\beta),\Re(\gamma), \Re(\delta) > 0,i=1,2,3$. Then the following holds:
\begin{equation*}
\int\limits_{0}^{r} s^{\delta-1} E_{\alpha,\beta,\gamma,\delta}^{\eta}(\lambda_1 s^{\alpha},\lambda_2 s^{\beta},\lambda_3 s^\gamma) \mathrm{d}s=r^{\delta} E_{\alpha,\beta,\gamma,\delta+1}^{\eta}(\lambda_1r^{\alpha},\lambda_2r^{\beta},\lambda_3r^\gamma).
\end{equation*}
\end{coroll}

Next we consider the R--L type fractional integral and derivative of a trivariate M--L function of order $\alpha \in \mathbb{C}$ where $\Re(\alpha)>0$.
\begin{thm}
Let $a \in \mathbb{R_{+}},\nu,\alpha,\beta,\gamma,\delta,\eta,\lambda_i \in \mathbb{C}$ with $\Re(\nu),\Re(\alpha),\Re(\beta),\Re(\gamma),\Re(\delta) > 0,i=1,2,3$.Then for $ r,y > a $, there holds the following relations:
\begin{align}\nonumber
&\left( \prescript{}{}I^{\nu}_{a_+}\left( (r-a)^{\delta-1} E_{\alpha,\beta,\gamma, \delta}^{\eta}(\lambda_{1}(r-a)^{\alpha},\lambda_{2}(r-a)^{\beta}, \lambda_{3}(r-a)^{\gamma})\right)\right)(y) \\
&=(y-a)^{\delta +\nu-1}E_{\alpha,\beta,\gamma, \delta +\nu}^{\eta}(\lambda_{1}(y-a)^{\alpha},\lambda_{2}(y-a)^{\beta}, \lambda_{3}(y-a)^{\gamma}),
\end{align}
and
\begin{align}\nonumber \label{B}
&\left( \prescript{}{}D^{\nu}_{a_+}\left( (r-a)^{\delta-1} E_{\alpha,\beta,\gamma, \delta}^{\eta}(\lambda_{1}(r-a)^{\alpha},\lambda_{2}(r-a)^{\beta}, \lambda_{3}(r-a)^{\gamma})\right) \right)(y) \\
&=(y-a)^{\delta -\nu-1}E_{\alpha,\beta,\gamma, \delta-\nu}^{\eta}(\lambda_{1}(y-a)^{\alpha},\lambda_{2}(y-a)^{\beta}, \lambda_{3}(y-a)^{\gamma}).
\end{align}
\end{thm}
\begin{proof}
By the aid of formulas \eqref{R-L integral} and \eqref{trivtype} with the following relation (\cite{Samko et al.}, Eq. (2.44))
\begin{equation*}
\left( \prescript{}{}I^{\nu}_{a_+} (r-a)^{\gamma-1}\right) (y)=\frac{\Gamma(\gamma)}{\Gamma(\gamma+\nu)}(y-a)^{\gamma+\nu-1},
\end{equation*} 
where $\alpha,\beta \in \mathbb{C}$ with $\Re(\alpha),\Re(\beta)>0$, leads to
\begin{align*}\nonumber
&\left( \prescript{}{}I^{\nu}_{a_+}\left( (r-a)^{\delta-1} E_{\alpha,\beta,\gamma, \delta}^{\eta}(\lambda_{1}(r-a)^{\alpha},\lambda_{2}(r-a)^{\beta}, \lambda_{3}(r-a)^{\gamma})\right)\right)(y) \\
&=\sum_{l=0}^{\infty}\sum_{p=0}^{\infty}\sum_{k=0}^{\infty}\frac{(\eta)_{l+p+k}}{\Gamma(l\alpha+p\beta+k\gamma+\delta)}\frac{\lambda_{1}^l\lambda_{2}^p\lambda_{3}^k}{l!p!k!}\left( \prescript{}{}I^{\nu}_{a_+} (r^{l\alpha+p\beta+k\gamma+\delta-1})\right) (y)  \\
&=(y-a)^{\delta +\nu-1}E_{\alpha,\beta,\gamma, \delta +\nu}^{\eta}(\lambda_{1}(y-a)^{\alpha},\lambda_{2}(y-a)^{\beta}, \lambda_{3}(y-a)^{\gamma}),  \quad \text{for}\quad r,y > a.
\end{align*}

To prove \eqref{B}, we are using the results of \eqref{R-L derivative} and \eqref{trivtype}  obtain that
\begin{align*}
&\left( \prescript{}{}D^{\nu}_{a_+}\left( (r-a)^{\delta-1} E_{\alpha,\beta,\gamma, \delta}^{\eta}(\lambda_{1}(r-a)^{\alpha},\lambda_{2}(r-a)^{\beta}, \lambda_{3}(r-a)^{\gamma}) \right)\right) (y) \\
=&\left(\frac{d}{dy}\right)^{n} \left( \prescript{}{}I^{n-\nu}_{a_+}\left( (r-a)^{\delta-1} E_{\alpha,\beta,\gamma, \delta}^{\eta}(\lambda_{1}(r-a)^{\alpha},\lambda_{2}(r-a)^{\beta}, \lambda_{3}(r-a)^{\gamma}) \right)\right) (y)\\ =&\left(\frac{d}{dy}\right)^{n} (y-a)^{\delta +n-\nu-1}E_{\alpha,\beta,\gamma, \delta+n-\nu}^{\eta}(\lambda_{1}(y-a)^{\alpha},\lambda_{2}(y-a)^{\beta}, \lambda_{3}(y-a)^{\gamma}).
\end{align*}
Result \eqref{B} is obtained by the virtue of \eqref{ordinary derivative}.
\end{proof}

Next we will consider the fractional derivative of M--L type function with three variable in Caputo's sense.
\begin{lem}\label{lem1}
Suppose that $\gamma, \nu \in \mathbb{C}$ with $\Re(\nu) \geq 0$. Then Caputo fractional differentiation of $\frac{(r-a)^{\gamma}}{\Gamma(\gamma+1)}$  is given by:
\begin{equation*}
\prescript{C}{}D^{\nu}_{a_+}\left( \frac{(r-a)^{\gamma}}{\Gamma(\gamma+1)}\right)(y)=\frac{(y-a)^{\gamma-\nu}}{\Gamma(\gamma-\nu+1)}, \quad r,y >a.
\end{equation*}
\end{lem}
\begin{proof}
	Let $n \coloneqq \lfloor \Re(\nu) \rfloor  +1$, $\Re(\nu)\geq 0$. Then we have: 
	\begin{align*}
	\prescript{C}{}D^{\nu}_{a_+}\left( \frac{(r-a)^{\gamma}}{\Gamma(\gamma+1)}\right)(y)&=\prescript{}{}I^{n-\nu}_{a_+}\left( \frac{d^{n}}{dr^{n}}\frac{(r-a)^{\gamma}}{\Gamma(\gamma+1)} \right)(y)
	=\prescript{}{}I^{n-\nu}_{a_+}\left( \frac{(r-a)^{\gamma-n}}{\Gamma(\gamma-n+1)} \right)(y)\\
	&=\frac{1}{\Gamma(n-\nu)}\int_{a}^{y}(y-s)^{n-\nu-1}\frac{(s-a)^{\gamma-n}}{\Gamma(\gamma-n+1)}\mathrm{d}s\\
	&\overset{u=\frac{y-s}{y-a}}{=}\frac{1}{\Gamma(n-\nu)}\frac{(y-a)^{\gamma-\nu}}{\Gamma(\gamma-n+1)}\mathbf{B}(n-\nu, \gamma-n+1)\\
	&=\frac{(y-a)^{\gamma-\nu}}{\Gamma(\gamma-\nu+1)}.
	\end{align*}
\end{proof}
\begin{thm}
	Let $\lambda_{i}\in \mathbb{C}, i=1,2,3$, $\nu,\alpha, \beta, \gamma, \delta, \eta \in \mathbb{C}$ with $\Re(\nu) \geq 0$,$\Re(\alpha)>0$, $\Re(\beta)>0$, $\Re(\gamma)>0$ and $\Re(\delta)>0$. Then the fractional differentiation of the function \eqref{trivtype} in Caputo sense is given by:
	\begin{align*}
	&\left( \prescript{C}{}D^{\nu}_{a_+}\left( (r-a)^{\delta-1} E_{\alpha,\beta,\gamma, \delta}^{\eta}(\lambda_{1}(r-a)^{\alpha},\lambda_{2}(r-a)^{\beta}, \lambda_{3}(r-a)^{\gamma}) \right)\right) (y) \\
	&=(y-a)^{\delta -\nu-1}E_{\alpha,\beta,\gamma, \delta-\nu}^{\eta}(\lambda_{1}(y-a)^{\alpha},\lambda_{2}(y-a)^{\beta}, \lambda_{3}(y-a)^{\gamma}), \quad r,y > a.
	\end{align*}
\end{thm}
\begin{proof}
    According to Lemma \ref{lem1}, fractionally differentiating the series \eqref{trivtype} term by term gives:
	\begin{align*}
	&\left(\prescript{C}{}D^{\nu}_{a_{+}}\left( (r-a)^{\delta-1} E_{\alpha,\beta,\gamma, \delta}^{\eta}(\lambda_{1}(r-a)^{\alpha},\lambda_{2}(r-a)^{\beta}, \lambda_{3}(r-a)^{\gamma}) \right)\right) (y)\\
	&=\sum_{l=0}^{\infty}\sum_{p=0}^{\infty}\sum_{k=0}^{\infty}\frac{(\eta)_{l+p+k}}{l!p!k!}\prescript{C}{}D^{\nu}_{a_{+}}\left(\frac{(r-a)^{l\alpha+p\beta+k\gamma+\delta-1}}{\Gamma(l\alpha+p\beta+k\gamma+\delta)} \right)(y)\\
	&=(y-a)^{\delta-\nu-1} \sum_{l=0}^{\infty}\sum_{p=0}^{\infty}\sum_{k=0}^{\infty}\frac{(\eta)_{l+p+k}}{\Gamma(l\alpha+p\beta+k\gamma+\delta-\nu)}\frac{(\lambda_{1}(y-a)^{\alpha})^{l}}{l!}\frac{(\lambda_{2}(y-a)^{\beta})^{n}}{p!}\frac{(\lambda_{3}(y-a)^{\gamma})^{k}}{k!}\\
	&=(y-a)^{\delta-\nu-1}E_{\alpha,\beta,\gamma, \delta-\nu}^{\eta}(\lambda_{1}(y-a)^{\alpha},\lambda_{2}(y-a)^{\beta}, \lambda_{3}(y-a)^{\gamma}),
	\end{align*} 
	where the result is also uniformly convergent.
\end{proof}

\section{Multi-term fractional differential equations}\label{sec:multi}
In this section, we derive an explicit solutions to homogeneous and inhomogenous multi-term FDEs.
\subsection{Analytical representation of solution to the homogeneous multi-term fractional differential equation}
In this subsection, we consider the initial value problem for linear homogeneous FDE with three independent fractional orders: 
\begin{equation} \label{uniBiv}
\left( \prescript{C}{}D^{\alpha}_{0+}y\right) (r)-\lambda_{3}\left( \prescript{C}{}D^{\beta}_{0+}y\right) (r)-\lambda_{2}\left( \prescript{C}{}D^{\gamma}_{0+}y\right) (r)-\lambda_{1}y(r)= 0,
\end{equation}
with initial condition $ y(0) = y_{0}$.

The following lemma and trinomial identity will be of significance for our results in the next theorem.

\begin{lem} \label{Lem:biML}
	For any parameters $\rho,\alpha,\beta,\gamma,\delta,\lambda_1,\lambda_2, \lambda_{3} \in\mathbb{R}$ satisfying $\rho \geq 0,\alpha,\beta,\gamma >0$ and $\delta-1>\lfloor\rho\rfloor$, we have:
	\[
	\left( \prescript{C}{}D^{\rho}_{0+}\Big[r^{\delta-1}E_{\alpha,\beta,\gamma,\delta}(\lambda_1r^{\alpha},\lambda_2r^{\beta}, \lambda_{3}r^{\gamma})\Big]\right)(y) =y^{\delta-\rho-1}E_{\alpha,\beta,\gamma,\delta-\rho}(\lambda_{1} y^{\alpha},\lambda_{2} y^{\beta},\lambda_{3} y^{\gamma}),\quad r,y > 0.
	\]
\end{lem}
\begin{proof}
	From Lemma \ref{lem1}, we have
	\[ \label {C}
	\left( \prescript{C}{}D^{\nu}_{0+}(\frac{r^{\mu}}{\Gamma(\mu+1)})\right)(y)=\frac{y^{\mu-\nu}}{\Gamma(\mu-\nu+1)},\quad\mu>\lfloor\nu\rfloor,\quad r,y > 0.
	\]
	Therefore, in accordance with \eqref{C}, fractionally differentiating the function \eqref{trivtype} term by term :
	\begin{align*}
	&\left( \prescript{C}{}D^{\rho}_{0+}\left[\sum_{l=0}^{\infty}\sum_{p=0}^{\infty}\sum_{k=0}^{\infty}\frac{(l+p+k)!\lambda_1^l\lambda_2^p\lambda_3^kr^{l\alpha+p\beta+k\gamma+\delta-1}}{\Gamma(l\alpha+p\beta+k\gamma+\delta)l!p!k!}\right]\right)(y)\\
	&=\sum_{l=0}^{\infty}\sum_{p=0}^{\infty}\sum_{k=0}^{\infty}\frac{(l+p+k)!\lambda_{1}^l\lambda_2^p\lambda_3^k y^{l\alpha+p\beta+k\gamma+\delta-\rho-1}}{\Gamma(l\alpha+p\beta+k\gamma+\delta-\rho)l!p!k!}\\
	&=y^{\delta-\rho-1}E_{\alpha,\beta,\gamma,\delta-\rho}(\lambda_{1} y^{\alpha},\lambda_{2} y^{\beta},\lambda_{3} y^{\gamma}),\quad r,y > 0.
	\end{align*}
\end{proof}

\textbf{Pascal's tetrahedron}.  If $q \geq 1$, $lpk \neq 0$, then
\begin{equation}\label{pascal}
\binom{q}{l,p,k}=\binom{q-1}{l-1,p,k}+\binom{q-1}{l,p-1,k}+\binom{q-1}{l,p,k-1}.
\end{equation}
In other case, the so-called Pascal's rule holds, for example $k=0$ and $lp \neq 0$
\begin{equation*}
\binom{q}{l,p}=\binom{q-1}{l-1,p}+\binom{q-1}{l,p-1}.
\end{equation*}
If $q= l+p+k$, then trinomial coefficient is defined by
\begin{equation*}
\binom{l+p+k}{l,p,k}=\frac{(l+p+k)!}{l!p!k!}.
\end{equation*}

\begin{thm} \label{unithm}
	The univariate form \eqref{trivtype} of the trivariate M--L function \eqref{tML4} with $\eta=1$, gives a solution 
	\begin{equation} \label{solution of homogenous case}
	y(r)= \left( 1+ \lambda_{1}r^{\alpha}E_{\alpha,\alpha-\gamma,\alpha-\beta, \alpha+1}(\lambda_{1}r^{\alpha}, \lambda_{2}r^{\alpha-\gamma}, \lambda_{3}r^{\alpha-\beta})\right)y_0,
	\end{equation}
for  intial value problem for the multi-term differential equation involving three independent fractional orders \eqref{uniBiv}.

\end{thm}
\begin{proof}
	It should be noted that the Caputo derivative of constant function is equal to zero. We will apply Lemma \ref{Lem:biML} to show that $y(r)$ is a solution of \eqref{uniBiv}. Starting from the series \eqref{trivtype}, we evaluate the fractional differ-integrals of $y(r)$ as below:
	\allowdisplaybreaks
	\begin{align*}
		\left(\prescript{C}{}D^{\alpha}_{0+}y\right) (r)&= \prescript{C}{}D^{\alpha}_{0+}\left(1+\sum_{l=0}^{\infty}\sum_{p=0}^{\infty}\sum_{k=0}^{\infty}\binom{l+p+k}{l,p,k}\frac{\lambda_{1}^{l+1}\lambda_2^p\lambda_3^kr^{l\alpha+p(\alpha-\gamma)+k(\alpha-\beta)+\alpha}}{\Gamma(l\alpha+p(\alpha-\gamma)+k(\alpha-\beta)+\alpha+1)} \right) y_0 \\
		&=\sum_{l=0}^{\infty}\sum_{p=0}^{\infty}\sum_{k=0}^{\infty}\binom{l+p+k}{l,p,k}\lambda_{1}^{l+1} \lambda_2^p\lambda_3^k\prescript{C}{}D^{\alpha}_{0+}\left(\frac{r^{l\alpha+p(\alpha-\gamma)+k(\alpha-\beta)+\alpha}}{\Gamma(l\alpha+p(\alpha-\gamma)+k(\alpha-\beta)+\alpha+1)} \right)y_0 \\
		&=\sum_{l=0}^{\infty}\sum_{p=0}^{\infty}\sum_{k=0}^{\infty}\binom{l+p+k}{l,p,k} \frac{\lambda_{1}^{l+1} \lambda_2^p\lambda_3^kr^{l\alpha+p(\alpha-\gamma)+k(\alpha-\beta)}}{\Gamma(l\alpha+p(\alpha-\gamma)+k(\alpha-\beta)+1)}y_0\\
		=\Big( \lambda_{1}&+\sum_{l=1}^{\infty}\sum_{p=0}^{\infty}\sum_{k=0}^{\infty}\binom{l+p+k-1}{l-1,p,k}\frac{\lambda_{1}^{l+1} \lambda_2^p\lambda_3^kr^{l\alpha+p(\alpha-\gamma)+k(\alpha-\beta)}}{\Gamma(l\alpha+p(\alpha-\gamma)+k(\alpha-\beta)+1)}\\
		&+\sum_{l=0}^{\infty}\sum_{p=1}^{\infty}\sum_{k=0}^{\infty}\binom{l+p+k-1}{l,p-1,k}\frac{\lambda_{1}^{l+1} \lambda_2^p\lambda_3^kr^{l\alpha+p(\alpha-\gamma)+k(\alpha-\beta)}}{\Gamma(l\alpha+p(\alpha-\gamma)+k(\alpha-\beta)+1)}\\
		&+\sum_{l=0}^{\infty}\sum_{p=0}^{\infty}\sum_{k=1}^{\infty}\binom{l+p+k-1}{l,p,k-1}\frac{\lambda_{1}^{l+1} \lambda_2^p\lambda_3^kr^{l\alpha+p(\alpha-\gamma)+k(\alpha-\beta)}}{\Gamma(l\alpha+p(\alpha-\gamma)+k(\alpha-\beta)+1)}\Big)y_0
     \end{align*}
Similarly, we have :
\begin{align*}
\lambda_{3}\left( \prescript{C}{}D^{\beta}_{0+}y\right) (r)&=\lambda_{3}\prescript{C}{}D^{\beta}_{0+}\Big(1+ \lambda_{1}r^{\alpha}E_{\alpha,\alpha-\gamma,\alpha-\beta, \alpha+1}(\lambda_{1}r^{\alpha}, \lambda_{2}r^{\alpha-\gamma}, \lambda_{3}r^{\alpha-\beta})\Big) y_0 \\
&=\lambda_{1}\lambda_{3}r^{\alpha-\beta}E_{\alpha,\alpha-\gamma,\alpha-\beta, \alpha-\beta+1}(\lambda_{1}r^{\alpha}, \lambda_{2}r^{\alpha-\gamma}, \lambda_{3}r^{\alpha-\beta})y_0\\
&=\sum_{l=0}^{\infty}\sum_{p=0}^{\infty}\sum_{k=0}^{\infty}\binom{l+p+k}{l,p,k}\frac{\lambda_{1}^{l+1} \lambda_2^p\lambda_3^{k+1}r^{l\alpha+p(\alpha-\gamma)+(k+1)(\alpha-\beta)}}{\Gamma(l\alpha+p(\alpha-\gamma)+(k+1)(\alpha-\beta)+1)}y_0\\
&=\sum_{l=0}^{\infty}\sum_{p=0}^{\infty}\sum_{k=1}^{\infty}\binom{l+p+k-1}{l,p,k-1}\frac{\lambda_{1}^{l+1} \lambda_2^p\lambda_3^kr^{l\alpha+p(\alpha-\gamma)+k(\alpha-\beta)}}{\Gamma(l\alpha+p(\alpha-\gamma)+k(\alpha-\beta)+1)}y_0,
\end{align*}
and 
\begin{align*}
\lambda_{2} \left(\prescript{C}{}D^{\gamma}_{0+}y \right) (r)&=\lambda_{2} \prescript{C}{}D^{\gamma}_{0+}\Big(1+ \lambda_{1}r^{\alpha}E_{\alpha,\alpha-\gamma,\alpha-\beta, \alpha+1}(\lambda_{1}r^{\alpha}, \lambda_{2}r^{\alpha-\gamma}, \lambda_{3}r^{\alpha-\beta})\Big)y_{0} \\
&=\lambda_{1}\lambda_{2}r^{\alpha-\gamma}E_{\alpha,\alpha-\gamma,\alpha-\beta, \alpha-\gamma+1}(\lambda_{1}r^{\alpha}, \lambda_{2}r^{\alpha-\gamma}, \lambda_{3}r^{\alpha-\gamma})y_0\\
&=\sum_{l=0}^{\infty}\sum_{p=0}^{\infty}\sum_{k=0}^{\infty}\binom{l+p+k}{l,p,k}\frac{\lambda_{1}^{l+1} \lambda_2^{p+1}\lambda_3^{k}r^{l\alpha+(p+1)(\alpha-\gamma)+k(\alpha-\beta)}}{\Gamma(l\alpha+(p+1)(\alpha-\gamma)+k(\alpha-\beta)+1)}y_0\\
&=\sum_{l=0}^{\infty}\sum_{p=1}^{\infty}\sum_{k=0}^{\infty}\binom{l+p+k-1}{l,p-1,k}\frac{\lambda_{1}^{l+1} \lambda_2^{p}\lambda_3^{k}r^{l\alpha+p(\alpha-\gamma)+k(\alpha-\beta)}}{\Gamma(l\alpha+p(\alpha-\gamma)+k(\alpha-\beta)+1)}y_0,
\end{align*}
and 
\begin{align*}
\lambda_{1}y(r)&=\lambda_{1}\left( 1+ \lambda_{1}r^{\alpha}E_{\alpha,\alpha-\gamma,\alpha-\beta, \alpha+1}(\lambda_{1}r^{\alpha}, \lambda_{2}r^{\alpha-\gamma}, \lambda_{3}r^{\alpha-\beta})\right)  y_0  \\
&=\left( \lambda_{1}+\sum_{l=0}^{\infty}\sum_{p=0}^{\infty}\sum_{k=0}^{\infty}\binom{l+p+k}{l,p,k}\frac{\lambda_{1}^{l+2} \lambda_2^p\lambda_3^kr^{l\alpha+p(\alpha-\gamma)+k(\alpha-\beta)+ \alpha}}{\Gamma(l\alpha+p(\alpha-\gamma)+k(\alpha-\beta)+\alpha + 1)}\right) y_0\\
&=\left( \lambda_{1}+\sum_{l=1}^{\infty}\sum_{p=0}^{\infty}\sum_{k=0}^{\infty}\binom{l+p+k-1}{l-1,p,k}\frac{\lambda_{1}^{l+1} \lambda_2^p\lambda_3^kr^{l\alpha+p(\alpha-\gamma)+k(\alpha-\beta)}}{\Gamma(l\alpha+p(\alpha-\gamma)+k(\alpha-\beta)+1)}\right) y_0.
\end{align*}
Taking a linear combination, we find
\begin{align*}
&\left(\prescript{C}{}D^{\alpha}_{0+}y\right)(r)-\lambda_{3}\left( \prescript{C}{}D^{\beta}_{0+}y\right)(r)-\lambda_{2}\left( \prescript{C}{}D^{\gamma}_{0+}y\right)(r)-\lambda_{1}y(r)=0.
\end{align*}
which satisfying the initial data $y(0)=y_{0}$.
Thus, the results are proved.
\end{proof}
\begin{remark}
	The Cauchy problem \eqref{uniBiv}  has a solution given by using Fox-Wright functions \eqref{fox}
	\allowdisplaybreaks
	\begin{align*}
	y(r)&=\sum_{d=0}^{\infty}\Big(\sum_{l+p=d}\Big)\frac{\lambda^{l}_{1}\lambda^{p}_{2}}{l!p!}r^{(\alpha-\beta)d+\beta l+(\beta-\gamma)p}\Biggl\{ \prescript{}{1}{}{\Psi}_{1}\left[\begin{array}{ccc}
	(d+1,1) \\ ((\alpha-\beta)d+\beta l+(\beta-\gamma)p+1,\alpha-\beta) 
	\end{array} \Big| \lambda_{3}r^{\alpha-\beta}
	\right]\\
	&-\lambda_{3}r^{\alpha-\beta}\prescript{}{1}{}{\Psi}_{1}\left[\begin{array}{ccc}
	(d+1,1) \\ ((\alpha-\beta)(d+1)+\beta l+(\beta-\gamma)p+1,\alpha-\beta) 
	\end{array} \Big| \lambda_{3}r^{\alpha-\beta}
	\right]\\
	&-\lambda_{2}r^{\alpha-\gamma}\prescript{}{1}{}{\Psi}_{1}\left[\begin{array}{ccc}
	(d+1,1) \\ ((\alpha-\beta)d+\alpha-\gamma+\beta l+(\beta-\gamma)p+1,\alpha-\beta) 
	\end{array} \Big| \lambda_{3}r^{\alpha-\beta}
	\right] \Biggr\}y_{0}\\
	&=\Biggl\{\sum_{l=0}^{\infty}\sum_{p=0}^{\infty}\sum_{k=0}^{\infty}\frac{\lambda^{l}_{1}\lambda^{p}_{2}\lambda^{k}_{3}}{l!p!k!}\frac{\Gamma(l+p+k+1)r^{l\alpha+p(\alpha-\gamma)+k(\alpha-\beta)}}{\Gamma(l\alpha+p(\alpha-\gamma)+k(\alpha-\beta)+1)}\\
	&-\sum_{l=0}^{\infty}\sum_{p=0}^{\infty}\sum_{k=0}^{\infty}\frac{\lambda^{l}_{1}\lambda^{p}_{2}\lambda^{k+1}_{3}}{l!p!k!}\frac{\Gamma(l+p+k+1)r^{l\alpha+p(\alpha-\gamma)+(k+1)(\alpha-\beta)}}{\Gamma(l\alpha+p(\alpha-\gamma)+(k+1)(\alpha-\beta)+1)} \\
	&-\sum_{l=0}^{\infty}\sum_{p=0}^{\infty}\sum_{k=0}^{\infty}\frac{\lambda^{l}_{1}\lambda^{p+1}_{2}\lambda^{k}_{3}}{l!p!k!}\frac{\Gamma(l+p+k+1)r^{l\alpha+(p+1)(\alpha-\gamma)+k(\alpha-\beta)}}{\Gamma(l\alpha+(p+1)(\alpha-\gamma)+k(\alpha-\beta)+1)}\Biggr\}y_{0}\\
	&=\Biggl\{\sum_{l=0}^{\infty}\sum_{p=0}^{\infty}\sum_{k=0}^{\infty}\binom{l+p+k}{l,p,k}\frac{\lambda^{l}_{1}\lambda^{p}_{2}\lambda^{k}_{3}r^{l\alpha+p(\alpha-\gamma)+k(\alpha-\beta)}}{\Gamma(l\alpha+p(\alpha-\gamma)+k(\alpha-\beta)+1)}\\
	&-\sum_{l=0}^{\infty}\sum_{p=0}^{\infty}\sum_{k=0}^{\infty}\binom{l+p+k}{l,p,k}\frac{\lambda^{l}_{1}\lambda^{p}_{2}\lambda^{k+1}_{3}r^{l\alpha+p(\alpha-\gamma)+(k+1)(\alpha-\beta)}}{\Gamma(l\alpha+p(\alpha-\gamma)+(k+1)(\alpha-\beta)+1)}\\
	&-\sum_{l=0}^{\infty}\sum_{p=0}^{\infty}\sum_{k=0}^{\infty}\binom{l+p+k}{l,p,k}\frac{\lambda^{l}_{1}\lambda^{p+1}_{2}\lambda^{k}_{3}r^{l\alpha+(p+1)(\alpha-\gamma)+k(\alpha-\beta)}}{\Gamma(l\alpha+(p+1)(\alpha-\gamma)+k(\alpha-\beta)+1)}\Biggr\}y_{0}.
	\end{align*}
	Using Pascal's tetrahedron \eqref{pascal}, we derive the desired result:
	\allowdisplaybreaks
	\begin{align}
	y(r)&=\Big(\sum_{l=0}^{\infty}\sum_{p=0}^{\infty}\sum_{k=0}^{\infty}\binom{l+p+k}{l,p,k}\frac{\lambda^{l}_{1}\lambda^{p}_{2}\lambda^{k}_{3}r^{l\alpha+p(\alpha-\gamma)+k(\alpha-\beta)}}{\Gamma(l\alpha+p(\alpha-\gamma)+k(\alpha-\beta)+1)}\nonumber\\
	&-\sum_{l=0}^{\infty}\sum_{p=0}^{\infty}\sum_{k=1}^{\infty}\binom{l+p+k-1}{l,p,k-1}\frac{\lambda^{l}_{1}\lambda^{p}_{2}\lambda^{k}_{3}r^{l\alpha+p(\alpha-\gamma)+k(\alpha-\beta)}}{\Gamma(l\alpha+p(\alpha-\gamma)+k(\alpha-\beta)+1)}\nonumber\\
	&-\sum_{l=0}^{\infty}\sum_{p=1}^{\infty}\sum_{k=0}^{\infty}\binom{l+p+k-1}{l,p-1,k}\frac{\lambda^{l}_{1}\lambda^{p}_{2}\lambda^{k}_{3}r^{l\alpha+p(\alpha-\gamma)+k(\alpha-\beta)}}{\Gamma(l\alpha+p(\alpha-\gamma)+k(\alpha-\beta)+1)}\Big)y_{0}\nonumber\\
	&=\Big(1+\sum_{l=1}^{\infty}\sum_{p=0}^{\infty}\sum_{k=0}^{\infty}\binom{l+p+k-1}{l-1,p,k}\frac{\lambda^{l}_{1}\lambda^{p}_{2}\lambda^{k}_{3}r^{l\alpha+p(\alpha-\gamma)+k(\alpha-\beta)}}{\Gamma(l\alpha+p(\alpha-\gamma)+k(\alpha-\beta)+1)}\Big)y_{0}\nonumber\\
	&=\Big(1+\sum_{l=0}^{\infty}\sum_{p=0}^{\infty}\sum_{k=0}^{\infty}\binom{l+p+k}{l,p,k}\frac{\lambda^{l+1}_{1}\lambda^{p}_{2}\lambda^{k}_{3}r^{(l+1)\alpha+p(\alpha-\gamma)+k(\alpha-\beta)}}{\Gamma((l+1)\alpha+p(\alpha-\gamma)+k(\alpha-\beta)+1)}\Big)y_{0}\nonumber\\
	&=\Big(1+\lambda_{1}r^{\alpha}E_{\alpha,\alpha-\gamma,\alpha-\beta, \alpha+1}(\lambda_{1}r^{\alpha}, \lambda_{2}r^{\alpha-\gamma}, \lambda_{3}r^{\alpha-\beta})\Big)y_{0}.
	\end{align}
	Therefore, we show that the coincidence between our new results in terms of trivarite Mittag-Leffler type functions and the results shown in \cite{Kilbas} by means of generalized Wright functions.
\end{remark}
\subsection{Explicit solution of inhomogeneous differential equation with three fractional orders}
In this subsection, we investigate the exact analytical representation of solutions to linear inhomogeneous FDEs by the aid of the superposition principle to obtain solution of \eqref{uniBiv}. 

Consider the next two Caputo type multi-term FDEs with three independent orders, namely: inhomogeneous differential equation with homogeneous initial condition
\begin{equation} \label{3.1}
\begin{cases}
\left( \prescript{C}{}D^{\alpha}_{0+}y\right)(r)-\lambda_{3}\left( \prescript{C}{}D^{\beta}_{0+}y\right)(r)-\lambda_{2}\left( \prescript{C}{}D^{\gamma}_{0+}y\right)(r)-\lambda_{1}y(r)=g(r),\quad r > 0,\\
y(0)\equiv 0.
\end{cases}
\end{equation}
and homogeneous differential equation with inhomogeneous initial condition
\begin{equation} \label{3.2}
\begin{cases}
\left( \prescript{C}{}D^{\alpha}_{0+}y\right)(r)-\lambda_{3}\left( \prescript{C}{}D^{\beta}_{0+}y\right)(r)-\lambda_{2}\left( \prescript{C}{}D^{\gamma}_{0+}y\right)(r)-\lambda_{1}y(r)=0,\quad r > 0,\\
y(0)=y_0.
\end{cases}
\end{equation}
The next Lemma can be attained from classical ideas to get analytical solution of linear FDEs.
\begin{lem}
If $y_{1}(r)$ and $y_{2}(r)$ are the solutions of the problems \eqref{3.1} and \eqref{3.2}, respectively, then \\
$y(r)=y_{1}(r)+y_{2}(r)$ is the general solution of the Cauchy problem of \eqref{1.1}. 
\end{lem}
Notice that the solution $y_{2}(r)$ of \eqref{3.2} have studied in Section 3.1. Thus, to acquire our target we need to find $y_{1}(r)$ which is a particular solution of \eqref{1.1}.
\begin{thm}
A solution $\tilde{y}\in C^{1}([0, \infty), \mathbb{R})$ of \eqref{1.1} satisfying homogeneous initial data $y(0)\equiv 0$ has the following form
\begin{equation} \label{tildex}
\tilde{y}(r)= \int_{0}^{r}(r-s)^{\alpha-1}E_{\alpha,\alpha-\gamma,\alpha-\beta, \alpha}(\lambda_{1}(r-s)^{\alpha},\lambda_{2}(r-s)^{\alpha-\gamma}, \lambda_{3}(r-s)^{\alpha-\beta})g(s)\mathrm{d}s.
\end{equation}
\end{thm}
\begin{proof}
With the aid of the variation of constants method, every solution of inhomogeneous differential equation $\tilde{y}(r)$ should be hold as:
\begin{equation} \label{tildyy}
\tilde{y}(r)= \int_{0}^{r}(r-s)^{\alpha-1}E_{\alpha,\alpha-\gamma,\alpha-\beta, \alpha}(\lambda_{1}(r-s)^{\alpha},\lambda_{2}(r-s)^{\alpha-\gamma}, \lambda_{3}(r-s)^{\alpha-\beta})h(s)\mathrm{d}s,
\end{equation}
where $h(s), s\in[0,r]$ is an sought after scalar valued function which satisfying $\tilde{y}(0)=0$.

In accordance with Definition \ref{caputo} and applying Fubini's theorem for double integrals, we attain 
\allowdisplaybreaks
{\small
\begin{align*}
&(\prescript{C}{}D^{\alpha}_{0^{+}}\tilde{y})(r)
=(\prescript{}{}D^{\alpha}_{0^{+}}\tilde{y})(r)\\
&=\frac{1}{\Gamma(1-\alpha)}\frac{d}{dr}\int_{0}^{r}(r-s)^{-\alpha}\int_{0}^{s}(s-\tau)^{\alpha-1}E_{\alpha,\alpha-\gamma,\alpha-\beta, \alpha}(\lambda_1(s-\tau)^{\alpha}, \lambda_{2}(s-\tau)^{\alpha-\gamma}, \lambda_{3}(s-\tau)^{\alpha-\beta})h(\tau)d\tau ds\\
&=\frac{1}{\Gamma(1-\alpha)}\frac{d}{dr}\int_{0}^{r}\int_{0}^{s}(r-s)^{-\alpha}(s-\tau)^{\alpha-1}E_{\alpha,\alpha-\gamma,\alpha-\beta, \alpha}(\lambda_1(s-\tau)^{\alpha}, \lambda_{2}(s-\tau)^{\alpha-\gamma}, \lambda_{3}(s-\tau)^{\alpha-\beta})h(\tau)d\tau ds\\
&=\frac{1}{\Gamma(1-\alpha)}\frac{d}{dr}\int_{0}^{t}h(\tau)\left( \int_{\tau}^{r}(r-s)^{-\alpha}(s-\tau)^{\alpha-1}E_{\alpha,\alpha-\gamma,\alpha-\beta, \alpha}(\lambda_1(s-\tau)^{\alpha}, \lambda_{2}(s-\tau)^{\alpha-\gamma}, \lambda_{3}(s-\tau)^{\alpha-\beta})ds\right) d\tau\\
&=\frac{1}{\Gamma(1-\alpha)}\frac{d}{dr}\int_{0}^{r}h(\tau)\left( \int_{\tau}^{r}(r-s)^{-\alpha}(s-\tau)^{\alpha-1}\sum_{l,p,k=0}^{\infty}\binom{l+p+k}{l,p,k}\frac{\lambda_{1}^{l} \lambda_{2}^p\lambda_{3}^k (s-\tau)^{l\alpha+p(\alpha-\gamma)+k(\alpha-\beta)}}{\Gamma(l\alpha+p(\alpha-\gamma)+k(\alpha-\beta)+\alpha)}ds\right) d\tau\\
&=\frac{1}{\Gamma(1-\alpha)}\sum_{l,p,k=0}^{\infty}\binom{l+p+k}{l,p,k}\frac{d}{dr}\int_{0}^{r}\frac{\lambda_{1}^{l} \lambda_2^p\lambda_3^k (r-\tau)^{l\alpha+p(\alpha-\gamma)+k(\alpha-\beta)}}{\Gamma(l\alpha+p(\alpha-\gamma)+k(\alpha-\beta)+\alpha)}h(\tau)d\tau\\
&\times\mathbf{B}(1-\alpha,l\alpha+p(\alpha-\gamma)+k(\alpha-\beta)+\alpha)\\
&=\sum_{l=0}^{\infty}\sum_{p=0}^{\infty}\sum_{k=0}^{\infty}\binom{l+p+k}{l,p,k}\lambda_{1}^{l} \lambda_2^p\lambda_3^k\frac{d}{dr}\int_{0}^{r}\frac{(r-\tau)^{l\alpha+p(\alpha-\gamma)+k(\alpha-\beta)}}{\Gamma(l\alpha+p(\alpha-\gamma)+k(\alpha-\beta)+1)}h(\tau)d\tau\\
&=h(r)+\sum_{l=1}^{\infty}\sum_{p=0}^{\infty}\sum_{k=0}^{\infty}\binom{l+p+k-1}{l-1,p,k}\lambda_{1}^{l} \lambda_2^p\lambda_3^k \int_{0}^{r}\frac{(r-\tau)^{l\alpha+p(\alpha-\gamma)+k(\alpha-\beta)-1}}{\Gamma(l\alpha+p(\alpha-\gamma)+k(\alpha-\beta))}h(\tau)d\tau \\
&+\sum_{l=0}^{\infty}\sum_{p=1}^{\infty}\sum_{k=0}^{\infty}\binom{l+p+k-1}{l,p-1,k}\lambda_{1}^{l} \lambda_2^p\lambda_3^k \int_{0}^{r}\frac{(r-\tau)^{l\alpha+p(\alpha-\gamma)+k(\alpha-\beta)-1}}{\Gamma(l\alpha+p(\alpha-\gamma)+k(\alpha-\beta))}h(\tau)d\tau\\
&+\sum_{l=0}^{\infty}\sum_{p=0}^{\infty}\sum_{k=1}^{\infty}\binom{l+p+k-1}{l,p,k-1}\lambda_{1}^{l} \lambda_2^p\lambda_3^k \int_{0}^{r}\frac{(r-\tau)^{l\alpha+p(\alpha-\gamma)+k(\alpha-\beta)-1}}{\Gamma(l\alpha+p(\alpha-\gamma)+k(\alpha-\beta))}h(\tau)d\tau\\
&=h(r)+\sum_{l=0}^{\infty}\sum_{p=0}^{\infty}\sum_{k=0}^{\infty}\binom{l+p+k}{l,p,k}\lambda_{1}^{l+1} \lambda_2^p\lambda_3^k \int_{0}^{r}\frac{(r-\tau)^{(l+1)\alpha+p(\alpha-\gamma)+k(\alpha-\beta)-1}}{\Gamma((l+1)\alpha+p(\alpha-\gamma)+k(\alpha-\beta))}h(\tau)d\tau\\
&+\sum_{l=0}^{\infty}\sum_{p=0}^{\infty}\sum_{k=0}^{\infty}\binom{l+p+k}{l,p,k}\lambda_{1}^{l} \lambda_2^{n+1}\lambda_3^k \int_{0}^{r}\frac{(r-\tau)^{l\alpha+(p+1)(\alpha-\gamma)+k(\alpha-\beta)-1}}{\Gamma(l\alpha+(p+1)(\alpha-\gamma)+k(\alpha-\beta))}h(\tau)d\tau\\
&+\sum_{l=0}^{\infty}\sum_{p=0}^{\infty}\sum_{k=0}^{\infty}\binom{l+p+k}{l,p,k}\lambda_{1}^{l} \lambda_2^p\lambda_3^{k+1} \int_{0}^{r}\frac{(r-\tau)^{l\alpha+p(\alpha-\gamma)+(k+1)(\alpha-\beta)-1}}{\Gamma(l\alpha+p(\alpha-\gamma)+(k+1)(\alpha-\beta))}h(\tau)d\tau\\
&\coloneqq\lambda_{3}\left( \prescript{C}{}D^{\beta}_{0+}y \right)(r)+\lambda_{2}\left( \prescript{C}{}D^{\gamma}_{0+}y\right)(r)+ \lambda_{1}y(r) + g(r).
\end{align*}
}
So we achieve $h(r)=g(r)$ for $r>0$. The proof is complete.
\end{proof}
\begin{remark}
	The Cauchy problem for inhomogeneous equation \eqref{3.1} has a solution which is a particular solution of \eqref{1.1} given by
	\begin{equation*}
	\tilde{y}(r)=\int_{0}^{r}(r-s)^{\alpha-1}G_{\gamma,\beta,\alpha;\lambda_{3}}(r-s)g(s)\mathrm{d}s,
	\end{equation*}
	where 
	\begin{align*}
	G_{\gamma,\beta,\alpha;\lambda_{3}}(z)&=\sum_{d=0}^{\infty}\Big(\sum_{l+p=d}\Big)\frac{\lambda^{l}_{1}\lambda^{p}_{2}}{l!p!}z^{(\alpha-\beta)d+\beta l+(\beta-\gamma)p}\prescript{}{1}{}{\Psi}_{1}\left[\begin{array}{ccc}
	(d+1,1) \\ ((\alpha-\beta)d+\alpha+\beta l+(\beta-\gamma)p,\alpha-\beta) 
	\end{array} \Big| \lambda_{3}z^{\alpha-\beta}
	\right]\\
	&=\sum_{l=0}^{\infty}\sum_{p=0}^{\infty}\sum_{k=0}^{\infty}\frac{\lambda^{l}_{1}\lambda^{p}_{2}\lambda^{k}_{3}}{l!p!k!}\frac{\Gamma(l+p+k+1)z^{l\alpha+p(\alpha-\gamma)+k(\alpha-\beta)}}{\Gamma(l\alpha+p(\alpha-\gamma)+k(\alpha-\beta)+\alpha)}\\
	&=\sum_{l=0}^{\infty}\sum_{p=0}^{\infty}\sum_{k=0}^{\infty}\binom{l+p+k}{l,p,k}\frac{\lambda^{l}_{1}\lambda^{p}_{2}\lambda^{k}_{3}z^{l\alpha+p(\alpha-\gamma)+k(\alpha-\beta)}}{\Gamma(l\alpha+p(\alpha-\gamma)+k(\alpha-\beta)+\alpha)}\\
	&=E_{\alpha,\alpha-\gamma,\alpha-\beta, \alpha}(\lambda_{1}z^{\alpha}, \lambda_{2}z^{\alpha-\gamma}, \lambda_{3}z^{\alpha-\beta}).
	\end{align*}
	In other words, a particular solution can be represented  as follows:
	\begin{equation*}
	\tilde{x}(r)=\int_{0}^{r}(r-s)^{\alpha-1}E_{\alpha,\alpha-\gamma,\alpha-\beta, \alpha}(\lambda_{1}(r-s)^{\alpha}, \lambda_{2}(r-s)^{\alpha-\gamma}, \lambda_{3}(r-s)^{\alpha-\beta})g(s)\mathrm{d}s.
	\end{equation*}
\end{remark}
Thus, in inhomogeneous case we point out that the particular solution is as exactly same as the solution proved in \cite{Kilbas}.

The next theorem present the structure of representation for an exact analytical solutions to \eqref{1.1}. The proof of theorem is straightaway, so we omit it here.
\begin{thm} \label{superposition}
The analytical solution $y \in C^{1}([0,\infty), \mathbb{R})$ of \eqref{1.1} has the following formula:
\begin{align}\nonumber
y(r)&= \Big(1+ \lambda_{1}r^{\alpha}E_{\alpha,\alpha-\gamma,\alpha-\beta, \alpha+1}(\lambda_{1}r^{\alpha}, \lambda_{2}r^{\alpha-\gamma}, \lambda_{3}r^{\alpha-\beta})\Big)y_0 \\ &+\int_{0}^{r}(r-s)^{\alpha-1}E_{\alpha,\alpha-\gamma,\alpha-\beta, \alpha}(\lambda_{1}(r-s)^{\alpha}, \lambda_{2}(r-s)^{\alpha-\gamma}, \lambda_{3}(r-s)^{\alpha-\beta})g(s)\mathrm{d}s.
	\end{align}
\end{thm}

\section{An illustrative example}\label{sec:ex}
To accomplish this paper, we provide an example to demonstrate the above mentioned results. Let $\alpha=0.8, \beta=0.6,\gamma=0.4$ and $\lambda_1=0.5,\lambda_2=3,\lambda_3=5$. Consider the following Cauchy type  problem for Caputo fractional multi-term FDEs with three independent fractional orders :
\begin{equation} \label{IVP}
\begin{cases}
\left( \prescript{C}{}D^{0.8}_{0^{+}}y\right) (r)-5(\prescript{C}{}D^{0.6}_{0^{+}}y)(r)-3\left( \prescript{C}{}D^{0.4}_{0^{+}}y\right) (r)- 0.5 y(r)= 0, \quad r>0,\\
y(0)=2.
\end{cases}
\end{equation}
Using by the explicit formula \eqref{solution of homogenous case} for the solution of \eqref{uniBiv} can be  represented via triple infinite series:
\begin{equation*}
y(r)= \Big( 1+ \lambda_{1}r^{\alpha}E_{\alpha,\alpha-\gamma,\alpha-\beta,\alpha+1}\left( \lambda_{1}r^{\alpha},\lambda_{2}r^{\alpha-\gamma},\lambda_{3}r^{\alpha-\beta}\right)\Big) y_0,
\end{equation*}
where $r^{\alpha}E_{\alpha,\alpha-\gamma,\alpha-\beta,\alpha+1}$ is the trivariate Mittag-Leffler type function which is given by as below:
\begin{align*}
&r^{\alpha}E_{\alpha,\alpha-\gamma,\alpha-\beta, \alpha+1}(\lambda_{1}r^{\alpha},\lambda_{2}r^{\alpha-\gamma}, \lambda_{3}r^{\alpha-\beta})\\
&=\sum_{l=0}^{\infty}\sum_{p=0}^{\infty}\sum_{k=0}^{\infty}\binom{l+p+k}{l,p,k}\frac{\lambda_{1}^l\lambda_{2}^p\lambda_{3}^k}{\Gamma(l\alpha+p(\alpha-\gamma)+k(\alpha-\beta)+\alpha+1)}r^{l\alpha+p(\alpha-\gamma)+k(\alpha-\beta)+\alpha}.
\end{align*}
Therefore, we can attain that the analytical solution $y(r)\in C^{1}([0,\infty),\mathbb{R})$ of the initial value problem \eqref{IVP} can be represented via newly defined trivariate Mittag-Leffler type function as below:
\begin{equation}\label{sol}
y(r)= 2+ r^{0.8}E_{0.8,0.4,0.2,1.8}\left( 0.5r^{0.8},3r^{0.4},5r^{0.2}\right).
\end{equation}
Now, we are going to illustrate example for the solution of $y(r) \in C^{1}([0,\infty),\mathbb{R})$ in \eqref{sol}.
\begin{figure}[H]
	\centering
	\includegraphics[width=0.6\linewidth]{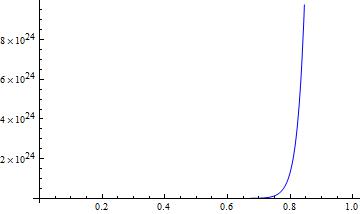}
	\caption{The graph of $y(r)$}
	\label{fig:fig4}
\end{figure}
\vspace{0.5cm}
\section{Conclusion}\label{sec:concl}
In this research work, we have proposed a new M--L function with three variables via  a triple infinite series of powers of $u$, $v$ and $w$ in the complex plane. The new trivariate M--L function arises from a number of various approaches, that motivates us to justify importance of these special functions. The advantage of this work is the solution of special case of  multi-term FDE involving three independent non-integer orders which can be extended to \cite{Joel et al.}-\cite{Al-Refai}. Meanwhile, the trivariate M--L function appears from certain applications in phsyics, e.g. electric circuit theory which can be expressed by means of the trivariate M--L function that will be discussed in the forthcoming paper. One can find the asymptotic expansion of the trivariate M--L function at $\infty$ by using the complex integral representation. Thus, the solution of FDEs system can be represented  in terms of the trivariate M--L functions which will be discussed in the forthcoming papers. 

Furthermore, one can except the results of this paper to hold for a class of problems such as Caputo type time-delay FDEs governed by
\begin{equation*} 
\begin{cases*}
\left( \prescript{C}{}D^{\alpha}_{0+}y\right)(r)-\lambda_{3}\left( \prescript{C}{}D^{\beta}_{0+}y\right)(r)-\lambda_{2}\left( \prescript{C}{}D^{\gamma}_{0+}y\right)(r)-\lambda_{1}y(r-h)=g(r),\quad r > 0, h >0,\\
y(r)=\varphi(r), \quad -h\leq r\leq 0.
\end{cases*}
\end{equation*}

   \end{document}